\documentclass[11pt,reqno]{amsart}
\usepackage{amsmath,a4wide,scalerel}
\usepackage{amsaddr}
\usepackage{mathabx}
\usepackage{xcolor,graphicx}
\usepackage{mathrsfs}
\usepackage{pgfplots}
\usepackage{subcaption}
\usepackage[normalem]{ulem}
\usepackage{float}
\usepackage{accents}
\usepackage[pagewise]{lineno}
\pgfplotsset{compat=1.18}
\usepackage{enumerate}
\usepackage{bbold}

\usepackage{hyperref}
\newtheorem{theorem}{Theorem}
\newtheorem{lemma}[theorem]{Lemma}
\newtheorem{proposition}[theorem]{Proposition}
\newtheorem{assumption}{Assumption}
\newtheorem{example}{Example}

\newcommand{\Z}{{\mathbb Z}}
\newcommand{\E}{{\mathbb E}}
\newcommand{\PP}{{\mathbb P}}
\newcommand{\WW}{{\mathbb W}}
\newcommand{\QQ}{{\mathbb Q}}

\newcommand{\EXP}{{\operatorname{exp}}}
\newcommand*\diff{\mathop{}\!\mathrm{d}}
\newcommand{\C}{\mathcal{C}}
\newcommand{\D}{\mathcal{D}}
\newcommand{\CD}{\operatorname{cl} \D}
\newcommand{\EM}{\textrm{EM}}
\newcommand{\SEM}{\textrm{SEM}}
\newcommand{\SEXP}{\textrm{EXP}}
\newcommand{\LTE}{\textrm{LTE}}

\title[]{Boundary-preserving weak Approximation for some semilinear stochastic partial differential equations}

\date{\today}

\author[]{Johan Ulander}

\begin{document}

\begin{abstract}
We propose and analyse a boundary-preserving numerical scheme for the weak approximation for some stochastic partial differential equations (SPDEs) with bounded state-space. We impose regularity assumptions on the drift and diffusion coefficients only locally on the state-space. In particular, the drift and diffusion coefficients may be non-globally Lipschitz continuous and superlinearly growing. The scheme consists of a finite difference discretisation in space and a Lie--Trotter time splitting followed by exact simulation and exact integration in time. The proposed scheme converges in the weak sense of order $1/4$ in time and of order $1/2$ in space, for globally Lipschitz continuous test functions. We prove the weak convergence order in time by proving strong convergence towards a strong solution driven by a different noise process. The convergence order in space follows from known results. The boundary-preserving property is ensured by the use of Lie--Trotter time splitting followed by exact simulation and exact integration. Numerical experiments confirm the theoretical results and demonstrate the practical advantages of the proposed Lie--Trotter-Exact (LTE) scheme compared to existing schemes for SPDEs.
\end{abstract}

\maketitle
\begin{sloppypar}



\section{Introduction}\label{sec:intro}

Stochastic differential equations (SDEs) arise in modelling various phenomena across scientific disciplines \cite{GraySIS, Karlin1981ASC,introTostocCalc,MR1214374,MR2001996}. Certain stochastic partial differential equations (SPDEs) take values only in a strict subset of the target-space, referred to as the invariant domain of the SPDE. For instance, the stochastic heat equation (SHE) with non-negative initial data should remain non-negative to be physically relevant. However, the SHEs with a globally Lipschitz source term and additive noise does not preserve non-negativity. 

Many SPDEs of interest result from adding noise to partial differential equations (PDEs) that take values only in a half-bounded or bounded domain. For instance, the solutions of the classical Allen--Cahn PDE take values only in $[-1,1]$. Introducing noise of the form $(1 - u^2) \diff W$ yields a stochastic Allen--Cahn equation whose solutions are also confined to $[-1,1]$. 

Classical numerical schemes for SPDEs are known to leave the invariant domains of SPDEs \cite{MR4729657}. In this work, we construct and analyse an implementable numerical scheme that preserves the invariant domain of the SPDE while converging in the weak sense to the mild solution.

We consider $(1+1)$-dimensional semilinear stochastic partial differential equations (SPDEs) of the form
\begin{equation}\label{intro:SPDE}
\left\lbrace
\begin{aligned}
& \diff u(t) = \left(\Delta u(t) + f(u(t)) \right) \diff t + g(u(t)) \diff W(t),\ (t,x) \in (0,T] \times [0,1], \\ 
& u(0,x) = u_{0}(x) \in \CD,\, x \in [0,1],
\end{aligned}
\right.
\end{equation}
where $\CD = [a,b]$ denotes the closure of some open and bounded domain $\D = (a,b) \subset \mathbb{R}$ and where $\diff W(t)$ is space-time white noise, subject to constant Dirichlet boundary conditions equal to a value in $\D$. The coefficient functions $f, g: \mathbb{R} \to \mathbb{R}$ are such that the mild solution $u(x,t) \in \CD$ for every $x \in [0,1]$ and $t \in [0,T]$, almost surely. Without loss of generality, we assume that $u$ satisfies homogeneous Dirichlet boundary conditions and that $0 \in \D$. See Section~\ref{sec:setting} for the precise definitions and assumptions. We remark that $f$ and $g$ are not assumed to be globally Lipschitz continuous. In fact, the numerical examples considered in this work (see Section~\ref{sec:num}) violate the global Lipschitz property in both $f$ and $g$. This lack of global Lipschitz continuity poses difficulties to approximate the solution of the SPDE in~\eqref{intro:SPDE}. Traditional SPDE theory requires $f$ and $g$ to be globally Lipschitz continuous to ensure well-posedness and convergence of classical numerical schemes. In fact, it was shown in \cite{MR2795791} that classical schemes may diverge in moments for SDEs violating the global Lipschitz property, and such SDEs may arise after discretising an SPDE in space.

A numerical scheme $\hat{u}_{m,n}$, for some indexing $m$ and $n$, for~\eqref{intro:SPDE} is called boundary-preserving if $\hat{u}_{m,n} \in \CD$ for all $m$ and $n$, almost surely. In many applications, the invariant domain of a model has physical significance, and numerical approximations that leave the invariant domain can be interpreted as non-physical. For example, population dynamics models \cite{GraySIS,KERMACK199133} and phase-separation dynamics models \cite{ALLEN19791085} are typically defined in $[0,1]$ and temperature models (e.g, the classical heat equation) are only physical in $[0,\infty)$. 

If a numerical approximation leaves the physical invariant domain of the model, then it can be influenced by forces that are not present in the model. To illustrate, considered the SPDE in~\eqref{intro:SPDE} with solutions confined to $\CD = [0,1]$. The solution remains unaffected by any modifications to the coefficient functions $f$ and $g$ outside $\CD=[0,1]$. However, modifying $f$ and $g$ outside $\CD = [0,1]$ can influence a numerical approximation that leaves the invariant domain $\CD = [0,1]$. Boundary-preserving numerical schemes are designed to avoid such discrepancies, ensuring that the numerical approximation remains in the relevant physical invariant domain.

The considered setting includes several important SPDE. For example, the stochastic Allen--Cahn equation \cite{Funaki1995TheSL} with $f(r) = r-r^3$ and $g(r) = 1 - r^2$, the stochastic Nagumo equation \cite{MCKEAN1970209} with $f(r)= r (1-r) (r - \gamma)$, for some $\gamma \in (0,1/2)$, and $g(r) = r (1-r)$, and the SIS SPDE with $f(r)= r ( 1-r)$ and $g(r) = r (1-r)$. For more details on these SPDEs and their properties, see Section~\ref{sec:num}, where we use them as test cases for the numerical experiments.

The proposed numerical scheme combines a finite difference discretisation in space with a Lie--Trotter splitting followed by exact integration and exact simulation in time. This ensures that the numerical scheme converges in the weak sense to the mild solution while preserving the invariant domain of the considered SPDE in~\eqref{intro:SPDE}. We abbreviate the scheme by $\LTE$ for \textbf{L}ie--\textbf{T}rotter and \textbf{E}xact integration and simulation.

The main results of the paper are the following:
\begin{itemize}
\item We propose a boundary-preserving numerical scheme for the weak approximation for SPDEs of the form~\eqref{intro:SPDE}.
\item We prove the boundary-preserving property and weak convergence of order $1/4$ in time for Lipschitz continuous test functions, see Proposition~\ref{prop:LTE-BP} and Theorem~\ref{th:mainWeak}.
\item We numerically verify the boundary-preserving property and weak convergence order of $1/4$ in time.
\end{itemize}
To the best of our knowledge, this is the first boundary-preserving numerical scheme for boundary-preserving SPDEs with coefficient functions that violate the global Lipschitz condition. 

The proof of weak convergence is based on establishing strong convergence to the mild solution of the SPDE~\eqref{intro:SPDE} with respect to a different driving noise. The change of the driving noise is necessary due to the use of exact simulation as part of the numerical time integration. For globally Lipschitz continuous test functions, this approach yields sharp order of weak convergence (see Section~\ref{sec:num}). It is important to note, however, that the proposed scheme does not convergence in the strong sense to the mild solution of the SPDE~\eqref{intro:SPDE}. This issue arises because of the use of exact simulation.

%
Before closing the introduction, we compare the proposed scheme to existing numerical schemes for similar problems with a focus on weakly convergent schemes and boundary-preserving schemes. The fully discrete numerical schemes consisting of FD approximation in space and EM, SEM, or exponential integrator approximation in time is known to convergence in the weak sense (in fact, the stronger property of convergence in probability) to the mild solution of~\eqref{intro:SPDE}. We refer to \cite{MR1644183,MR1699161} for the former and to \cite{MR4050540} for the latter. Although classical numerical schemes to approximate the solution of~\eqref{intro:SPDE} are known to convergence in the weak sense, such schemes are known to not be boundary-preserving \cite{MR4780408}. 

Boundary-preserving schemes for SDEs has been extensively studied over the last two decades, with approaches ranging from applying specific transforms (e.g., the Lamperti transform) to using truncation schemes. We mention the works \cite{MR4220738,MR4489741, MR4242953, MR2341800,MR3248050,artBar, MR4737060} as references to works on boundary-preserving schemes for SDEs. In contrast, boundary preservation for SPDEs has received significantly less attention. The first positivity-preserving scheme for some semilinear stochastic heat equations with space-dependent noise was proposed and studied in previous work by the author and collaborators in~\cite{MR4780408}. To the best of our knowledge, this is the second work on boundary-preserving numerical schemes for SPDEs and the first to treat SPDEs with bounded invariant domains, and allowing $f$ and $g$ to be superlinearly growing. 

This paper is organised as follows. In Section~\ref{sec:setting}, we specify the setting, assumptions, and the necessary preliminaries for later sections. Next, we describe the proposed numerical scheme, establish important properties  of it, and prove convergence in Section~\ref{sec:scheme}. Finally, we provide numerical experiments to support our theoretical findings in Section~\ref{sec:num}.

\section{Setting}\label{sec:setting}
In this section, we discuss the assumptions on the SPDE in~\eqref{intro:SPDE} that guarantees that the proposed scheme is well-defined, is boundary-preserving, and converges weakly to the mild solution. We consider a fixed probability space $(\Omega, \mathcal{F}, \mathbb{P})$ equipped with a filtration $\left( \mathcal{F}_{t} \right)_{t \in [0,T]}$ satisfying the usual conditions of completeness and right-continuity. We let $C$ denote a generic constant that may vary from line to line. To emphasise dependency on parameters $a_{1},\ldots,a_{j}$, with $j \in \mathbb{N}$, of interest, we use the notation $C(a_{1},\ldots,a_{j})$. The expectation operator is denoted by $\E$. We let $\C(\mathcal{A})$ denote the space of $\mathbb{R}$-valued continuous functions on $\mathcal{A}$ and we let $\C^{k}(\mathcal{A})$, with $k \in \mathbb{N}$, denote the space of $\mathbb{R}$-valued continuous functions on $\mathcal{A}$ with $k$ continuous derivatives. Most equalities and inequalities are to be interpreted in the almost sure sense, sometimes we omit to state this to avoid repetition.

We begin with describing the considered SPDEs and the assumptions imposed on the them in Section~\ref{sec:spde}. In this section, we also briefly discuss the boundary classifications of these SPDEs. Lastly, we provide the necessary background for the spatial and temporal discretisations, respectively. More precisely, the spatial discretisation using finite differences is described in Section~\ref{sec:spacedisc}, Section~\ref{sec:LTsplit} provides an overview of Lie--Trotter time splitting, and Section~\ref{sec:exactSim} is devoted to a construction of the exact simulation algorithm of SDEs.

\subsection{Description of the SPDEs}\label{sec:spde}
We start with discussing the necessary preliminaries and assumptions on the considered SPDE in~\eqref{intro:SPDE} for the construction and analysis of the proposed fully discrete numerical scheme.

We impose the following assumptions on the initial value. The regularity assumption on $u_{0}$ is the same as in \cite{MR4780408}.
\begin{assumption}\label{ass:u0}
The initial value $u_{0} \in \C^{3} \left( [0,1] \right)$, satisfies $u_{0}(x) \in [a,b]$ for every $x \in [0,1]$ and $u_{0}(0)=u_{0}(1)=0$.
\end{assumption}
Note that weaker regularity assumptions on $u_{0}$ could be used (but with lower convergence order in space), see for example \cite{MR4050540,MR1644183}. The condition $u_{0}(x) \in [a,b]$ for every $x \in [0,1]$ is needed for the mild solution $u$ to remain in $[a,b]$ for all times.

In the convergence proof, we need the drift and diffusion coefficients $f$ and $g$ to be locally Lipschitz functions and that the mild solution, the space-discrete FD solution and the fully discrete approximation all remain in the domain $\CD$ for all times. Moreover, we have to impose additional regularity assumptions on $f$ and $g$ to guarantee that exact simulation over time intervals of size $\Delta t$ is applicable. We use the result from \cite{MR2187299} to guarantee that exact simulation of the scalar time-homogeneous Itô-type SDEs
\begin{equation}\label{eq:SDE}
\left\lbrace
\begin{aligned}
& \diff X(t) = f(X(t)) \diff t + \lambda g(X(t)) \diff B(t),\ t \in (0,\Delta t], \\ 
& X(0) = x_{0} \in (a,b),
\end{aligned}
\right.
\end{equation}
for any $\lambda > 0$, on a small time interval $(0,\Delta t]$ is possible. Note that, by the assumptions on $f$ and $g$ below, the case $x_{0} \in \{a,b \}$ is trivial since then $X(t)=x_{0}$ for all times $t \in [0,T]$ (as $f$ and $g$ vanish on $\{ a,b \}$). In essence, the exact simulation algorithm requires that the Lamperti transform defined as
\begin{equation}\label{eq:LampTrans-sett}
\Phi(r) = \int_{r_{0}}^{r} \frac{1}{g(w)} \diff w,\ r \in (a,b),
\end{equation}
for any $r_{0} \in (a,b)$, is well-defined and that the following function
\begin{equation*}
\alpha(r) = \frac{f(\Phi^{-1}(r))}{g(\Phi^{-1}(r))} - \frac{1}{2} g'(\Phi^{-1}(r)),\ r \in \mathbb{R},
\end{equation*}
is well-defined and has certain regularity properties. Observe that the process defined by $Y(t) = \Phi \left( X(t) \right)$, by Itô's formula, then satisfies the following SDE with drift coefficient function $\alpha$ and with additive noise
\begin{equation}\label{eq:SDE-Lamp}
\left\lbrace
\begin{aligned}
& \diff Y(t) = \alpha(Y(t)) \diff t + \lambda \diff B(t),\ t \in (0,\Delta t], \\ 
& Y(0) = \Phi(x_{0}).
\end{aligned}
\right.
\end{equation}
We remark that we will transfer the constant diffusion coefficient $\lambda$ in~\eqref{eq:SDE-Lamp} to the drift coefficient by a deterministic time-changed. See Section~\ref{sec:exactSim} for more details on this. 

We impose the following conditions on $f$ and $g$ to guarantee that equations~\eqref{eq:SDE},~\eqref{eq:LampTrans-sett} and~\eqref{eq:SDE-Lamp} are well-defined.
\begin{assumption}\label{ass:f}
The drift coefficient $f \in \C^{2} \left([a,b]\right)$.
\end{assumption}

\begin{assumption}\label{ass:g}
The diffusion coefficient $g \in \C^{3} \left([a,b] \right)$ and is strictly positive on $(a,b)$,  and, for any $r_{0} \in (a,b)$, the following non-integrability conditions are satisfied
\begin{equation*}
\int_{r_{0}}^{a} \frac{1}{g(w)} \diff w = - \infty,\ \int_{r_{0}}^{b} \frac{1}{g(w)} \diff w = \infty.
\end{equation*}
\end{assumption}
Assumption~\ref{ass:g} is natural for any numerical scheme that utilises the Lamperti transform. Note that Assumption~\ref{ass:g} implies that $g(a)=g(b)=0$. We also remark that Assumption~\ref{ass:f} and Assumption~\ref{ass:g} only specifies $f$ and $g$ on $[a,b]$ and that the behaviour of $f$ and $g$ outside $[a,b]$ is of no importance. For instance, in Section~\ref{sec:num}, we provide numerical experiments for three SPDEs with superlinearly growing $f$ and $g$ that satisfy the assumptions in this section.

\begin{assumption}\label{ass:fg}
The drift coefficient $f$ decays at least as fast as the diffusion coefficient $g$ near the boundary points $\{a,b \}$; that is, the following limits exist and are finite
\begin{equation}\label{eq:limCond}
\left| \lim_{r \searrow a} \frac{f(r)}{g(r)} \right| + \left| \lim_{r \nearrow b} \frac{f(r)}{g(r)} \right| < \infty.
\end{equation}
\end{assumption}
Assumptions~\ref{ass:f},~\ref{ass:g} and~\ref{ass:fg} are used also in \cite{MR4737060} to guarantee that the Lamperti transform is well-defined and has desired properties. Note that Assumption~\ref{ass:fg} combined with Assumption~\ref{ass:g} implies that $f(a)=f(b)=0$. We also remark that Assumptions~\ref{ass:f},~\ref{ass:g} and~\ref{ass:fg} are sufficient, but possibly not necessary, assumptions.

The above assumptions implies regularity properties of $\Phi^{-1}$ and of $\alpha$ that we summarise in the following two propositions. The statements are taken from \cite{MR4737060} and we refer to \cite{MR4737060} for the proofs.
\begin{proposition}\label{prop:phiinv}
Suppose Assumpion~\ref{ass:g} is satisfied. Then $\Phi^{-1}: \mathbb{R} \to (a,b)$ is bounded, bijective, continuously differentiable and has bounded derivative. 
\end{proposition}

\begin{proposition}\label{prop:alpha}
Suppose Assumpions~\ref{ass:f},~\ref{ass:g} and~\ref{ass:fg} are satisfied. Then $\alpha \in \C^{2} \left(\mathbb{R}\right)$ and
\begin{equation*}
\sup_{r \in \mathbb{R}} |\alpha(r)| + \sup_{r \in \mathbb{R}} |\alpha'(r)| + \sup_{r \in \mathbb{R}} |\alpha''(r)| < \infty.
\end{equation*}
In particular, $\alpha, \alpha': \mathbb{R} \to \mathbb{R}$ are globally Lipschitz continuous.
\end{proposition}
Before moving on, we mention that, by the inverse function theorem and by the chain rule,
\begin{equation}\label{eq:diffPhiinv}
\frac{\diff}{\diff r} \Phi^{-1}(r) = g \left( \Phi^{-1}(r) \right).
\end{equation}
The relation~\eqref{eq:diffPhiinv} will be needed for the exact simulation implementation used for the numerical experiments in Section~\ref{sec:num}.

In this work, we consider mild solutions of the SPDE in~\eqref{intro:SPDE}. Therefore, we introduce the heat kernel with Dirichlet boundary conditions
\begin{equation}\label{eq:heatKer}
G(t,x,y) = 2 \sum_{j=1}^{\infty} e^{- j^{2} \pi^{2} t} \sin(j \pi x) \sin(j \pi y)
\end{equation}
associated with the heat equation subject to Dirichlet boundary conditions:
\begin{equation*}
\left\lbrace
\begin{aligned}
& \diff v(t) = \Delta v(t) \diff t,\ (t,x) \in (0,T] \times [0,1], \\ 
& v(0,x) = \delta_{0}(x),\, \forall x \in [0,1].
\end{aligned}
\right.
\end{equation*}
We say that $u$ is a mild solution of the SPDE in~\eqref{intro:SPDE} if it satisfies
\begin{equation}\label{eq:umild}
  \begin{split}
    u(t,x) &= \int_{0}^{1} G(t,x,y) u_{0}(y) dy \\ &+ \int_{0}^{t} \int_{0}^{1} G(t-s,x,y) f(u(s,y)) ds dy \\ &+ \int_{0}^{t} \int_{0}^{1} G(t-s,x,y) g(u(s,y)) \diff W(s,y),
  \end{split}
\end{equation}
for $(t,x) \in [0,T] \times [0,1]$, almost surely.

We now briefly discuss boundary classification for SPDEs, with a particular focus on the SPDEs in~\eqref{intro:SPDE}. In contrast to the extensive boundary classification for one-dimensional SDEs, the literature on boundary classification for SPDEs is not as well-established. One of the earliest results in this direction was \cite{MR1149348} where strict positivity was established for stochastic heat equations with non-negative initial values. In this section, we establish that the solution $u$ of the considered SPDE in~\eqref{intro:SPDE} remains in the domain $\CD = [a,b]$, hence referred to as the invariant domain of the SPDE. We state and prove this together with existence and uniqueness of a mild solution to the SPDE~\eqref{intro:SPDE} in the following proposition.
\begin{proposition}\label{prop:u-BP}
Suppose the assumptions in Section~\ref{sec:spde} are satisfied. Then there exists a unique continuous mild solution $u$ to the considered SPDE in~\eqref{intro:SPDE} and it satisfies
\begin{equation*}
\PP \left( u(t,x) \in \CD = [a,b],\ \forall t \in [0,T],\ \forall x \in [0,1] \right) = 1.
\end{equation*}
\end{proposition}
To the best of our knowledge, Proposition~\ref{prop:u-BP} has not been proved before and the question of whether or not the boundary values $\{a,b \}$ will be attain at later times $t \in (0,T]$ is not clear. We leave this complete boundary classification to future works. As we focus on the numerical analysis in this work, we provide the proof in Appendix~\ref{sec:app-EUB}. Note that all moments of $u$ are bounded by Proposition~\ref{prop:u-BP}.

\subsection{Spatial discretisation}\label{sec:spacedisc}
In this section, we provide the necessary background on finite difference approximations with uniform meshes for the specific case of the considered SPDE in~\eqref{intro:SPDE}. The presented material is essentially the same as in \cite{MR4780408,MR1644183}, we include it, in a slightly different style, for completeness. Recall that the spatial domain of the SPDE in~\eqref{intro:SPDE} is $[0,1]$. 

The finite difference method approximates differentials with finite difference quotients. We let $N \in \mathbb{N}$ and introduce the space grid size $\Delta x = 1/N$ and the corresponding space grid points $x_{n} = n \Delta x$, for $n = 0, \ldots,N$. Both $N$ and $\Delta x$ will be used for notational convenience. Associated with the introduced space grid, we let $\kappa^{N}:[0,1] \to \{ x_{0},\ldots, x_{N} \}$ be the step function defined by $\kappa^{N}(x) = x_{n}$ for $x \in [x_{n},x_{n+1})$, for $n=0,\ldots,N-1$, and $\kappa^{N}(1)=1$.

First, we discretise the initial value of the SPDE in~\eqref{intro:SPDE} by $u^{N}_{0,n} = u_{0}(\kappa^{N}(x_{n}))$, for $n=1,\ldots,N$. Let us denote by $u^{N}_{0} \in \mathbb{R}^{N-1}$ the vector with elements $\left( u^{N}_{0} \right)_{n} = u^{N}_{0,n}$. Next, we discretise the SPDE in~\eqref{intro:SPDE} in space by replacing differentials with finite difference quotients to obtain the following system of SDEs
\begin{align*}
\diff u^{N}(t,x_{n}) &= N^{2} \left( u^{N}(t,x_{n+1}) - 2 u^{N}(t,x_{n}) + u^{N}(t,x_{n-1}) \right) \diff t + f(u^{N}(t,x_{n})) \diff t \\ &+ N g(u^{N}(t,x_{n})) \diff \left( W(t,x_{n+1})-W(t,x_{n}) \right),
\end{align*}
for $n=1,\dots,N-1$, and
\begin{equation*}
u^{N}(t,0) = u^{N}(t,1) = 0.
\end{equation*}
We now re-write the above in a more compact form. Let us define the $\mathbb{R}^{N-1}$-valued stochastic processes $u^{N}(t)$ as the vector of size $N-1$ with elements $u^{N}_{n}(t) = \left(u^{N}(t) \right)_{n} = u^{N}(t,x_{n})$, for $n = 1,\ldots,N-1$ and for $t \in [0,T]$, and $W^{N}(t)$ as the vector of size $N-1$ with elements $W^{N}_{n}(t) = \left( W^{N}(t) \right)_{n} = \sqrt{N} \left( W(t,x_{n+1}) - W(t,x_{n}) \right)$, for $n = 1,\ldots,N-1$ and for $t \in [0,T]$. Note that we exclude the elements corresponding to $n=0$ and $n = N$, as $u^{N}(t,0) = u^{N}(t,1)=0$ by the Dirichlet boundary conditions. Then $W^{N}(t)$ is a vector containing $N-1$ independent standard Brownian motions and $u^{N}$ is the solution to the following system of SDEs in the Itô sense
\begin{equation}\label{eq:uNSDE}
\left\lbrace
\begin{aligned}
& \diff u^{N}(t) = N^{2} D^{N} u^{N}(t) \diff t + f(u^{N}(t)) \diff t + \sqrt{N} g(u^{N}(t)) \diff W^{N}(t),\ t \in (0,T], \\ 
& u^{N}(0) = u^{N}_{0},
\end{aligned}
\right.
\end{equation}
where $D^{N} = \left( D^{N}_{i,j} \right)_{1 \leq i,j, \leq N-1}$ denotes the standard finite difference matrix coming from the discretisation of the Laplacian with Dirichlet boundary conditions given by
\begin{equation*}
D^N=
\begin{pmatrix}
-2 & 1 & 0 & \ldots & 0 & 0 & 0\\
1 & -2 & 1 & \ddots & 0 & 0 & 0\\
0 & 1 & -2 & \ddots & 0 & 0 & 0\\
\vdots & \ddots & \ddots & \ddots & \ddots & \ddots & \vdots\\
0 & 0 & 0 & \ddots & -2 & 1 & 0\\
0 & 0 & 0 & \ddots & 1 & -2 & 1\\
0 & 0 & 0 & \ldots & 0 & 1 & -2
\end{pmatrix}.
\end{equation*}
We interpret $g(u^{N}(t)) \diff W^{N}(t)$ in~\eqref{eq:uNSDE} as a vector of size $N-1$ with elements $g(u^{N}_{n}(t)) \diff W^{N}_{n}(t)$.

For the purpose of the convergence proof, it will be convenient to define a piecewise linear extension $u^{N}: [0,T] \times [0,1] \to \mathbb{R}$ that satisfy an integral equation similar to that of $u$ (see equation~\eqref{eq:umild}). The idea comes from \cite{MR1644183}, but we include it for completeness. For $x \in [x_{n},x_{n+1})$, for any $n = 0,\ldots, N-1$, let
\begin{equation*}
u^{N}(t,x) = u^{N}_{n}(t) + (N x - n) \left( u^{N}_{n+1}(t) - u^{N}_{n}(t) \right)
\end{equation*}
and let $u^{N}(t,1) = u^{N}_{N}(t)=0$ (by homogeneous Dirichlet boundary conditions). The above two definitions of $u^{N}$ should not be confusing as they coincide on the space grid points. Finally, we introduce the kernel $G^{N}$ given by
\begin{equation}\label{eq:discKer}
G^{N}(t,x,y) = \sum_{j=1}^{N-1} e^{-\lambda_{j}^{N} t} \varphi_{j}^{N}(x) \varphi_{j}(\kappa^{N}(y)),
\end{equation} 
for $t \in [0,T]$ and for $x,y \in [0,1]$, where
\begin{equation*}
\lambda^{N}_{j} = 4 N \sin^{2} \left( \frac{j \pi}{2 N}  \right),\ \varphi_{j}(x) = \sqrt{2} \sin (j \pi x),
\end{equation*}
and where $\varphi^{N}_{j}$ is the linear interpolation of $\varphi_{j}$ at the space grid points $x_{0}, \ldots, x_{N}$. More precisely, $\varphi^{N}_{j}(x_{n}) = \varphi_{j}(x_{n})$ and
\begin{equation*}
\varphi^{N}_{j}(x) = \varphi_{j}(x_{n}) + (N x - n) \left( \varphi_{j}(x_{n+1}) - \varphi_{j}(x_{n}) \right)
\end{equation*}
for $x \in (x_{n},x_{n+1})$, for $n=1, \ldots, N-1$. Observe that the kernel $G^{N}$ defined in equation~\eqref{eq:discKer} resembles the heat kernel defined in equation~\eqref{eq:heatKer}. Also note that $\lambda_{j}^{N}$ and $(\varphi_{j}(x_{1}),\ldots,\varphi_{j}(x_{N-1}))$ is the $j$th eigenvalue-eigenvector pair of the matrix $D^{N}$. The kernel $G^{N}$ can also be viewing in a different way that we now discuss, which will be useful later on. Let us introduce the matrix $\left( G^{N}_{ij} \right)_{1 \leq i,j \leq N-1} = e^{t N^{2} D^{N}}$ for $t \in [0,T]$. Then one can show that $G^{N}(t,x_{i},_x{j}) = N G^{N}_{ij}(t)$, for $t \in [0,T]$. In other words, $G^{N}(t,x,y)$ is an extension to continuous space of the matrix exponential of $t N^{2} D^{N}$.

The kernel $G^{N}$ satisfies the following two estimates that will be used in the proof of weak convergence.
\begin{enumerate}
\item[\textbullet] For every $t \in (0,T]$ (see Lemma $2.3$ in \cite{MR4050540})
\begin{equation}\label{eq:intGNest}
\sup_{N \in \mathbb{N}} \sup_{x \in [0,1]} \int_{0}^{1} \left| G^{N}(t,x,y) \right|^{2} \diff y \leq C(T) \frac{1}{\sqrt{t}}.
\end{equation}
\item[\textbullet] For every $t \in (0,T]$ (see equation $(25b)$ in \cite{MR4780408})
\begin{equation}\label{eq:intGNdiffEst}
\sup_{N \in \mathbb{N}} \sup_{x \in [0,1]} \int_{0}^{t} \int_{0}^{1} \left| G^{N}(t-\ell^{M}(s),x,y) - G^{N}(t-s,x,y) \right|^{2} \diff s \diff y \leq \frac{\Delta t}{\Delta x}.
\end{equation}
\end{enumerate}
Similarly, we collect the following needed properties that the space-discrete approximation $u^{N}$ satisfies.
\begin{enumerate}
\item[\textbullet] $u^{N}$ satisfies the following integral equation (see equation $2.7$ in \cite{MR1644183})
\begin{equation}\label{eq:Nmild}
  \begin{split}
    u^{N}(t,x) &= \int_{0}^{1} G^{N}(t,x,y) u_{0}(\kappa^{N}(y)) \diff y \\ &+ \int_{0}^{t} \int_{0}^{1} G^{N}(t-s,x,y) f(u^{N}(s,\kappa^{N}(y))) \diff s \diff y \\ &+ \int_{0}^{t} \int_{0}^{1} G^{N}(t-s,x,y) g(u^{N}(s,\kappa^{N}(y))) \diff W(s,y),
  \end{split}
\end{equation}
for $s,t \in [0,T]$ and for $x,y \in [0,1]$.
\item[\textbullet] $u^{N}$ satisfies the following time regularity estimate (see Proposition $3$ in \cite{MR4780408})
\begin{equation}\label{eq:NstReg}
\E \left[ \left| u^{N}(s,x) - u^{N}(t,x) \right|^{2} \right] \leq C(T,f,g) \frac{|t-s|}{\Delta x}
\end{equation}
for $s,t \in (0,T]$ and for $x,y \in [0,1]$, almost surely.
\item[\textbullet] The difference between $u^{N}$ and the mild solution $u$ satisfies (see Theorem $3.1$ in \cite{MR1644183})
\begin{equation}\label{eq:Nstrong}
\sup_{t \in [0,T]} \sup_{x \in [0,1]} \left( \E \left[ \left| u^{N}(t,x) - u(t,x) \right|^{2} \right] \right)^{1/2} \leq C(T,f,g) \Delta x^{1/2}.
\end{equation}
\end{enumerate}
The three properties of $u^{N}$ stated above and the following proposition giving boundary preservation of $u^{N}$ will be used to prove weak convergence of the proposed $\LTE$ scheme. We provide the proof of Proposition~\ref{prop:N-BP} in Appendix~\ref{sec:app-uN}.
\begin{proposition}\label{prop:N-BP}
Suppose that the assumptions in Section~\ref{sec:spde} are satisfied. Let $u^{N}$ be the finite difference approximation with initial value $u^{N}(0,x) = u_{0}(\kappa^{N}(x))$. Then
\begin{equation*}
\PP \left( u^{N}(t,x) \in \CD = [a,b],\ \forall t \in [0,T],\ \forall x \in [0,1] \right) = 1.
\end{equation*}
\end{proposition}

We remark that Theorem~$3.1$ in \cite{MR1644183} cannot be directly applied to obtain the strong convergence in~\eqref{eq:Nstrong} as the drift and diffusion coefficients $f$ and $g$ are not globally Lipschitz continuous in the present setting. But, thanks to Proposition~\ref{prop:u-BP} and Proposition~\ref{prop:N-BP}, we can modify the coefficients $f$ and $g$ outside of $\CD = [a,b]$, without affecting the solution $u$ and the approximation $u^{N}$, to obtain an SPDE with globally Lipschitz continuous coefficients. We also remark that~\eqref{eq:Nstrong} requires $u_{0} \in \C^{3} \left( [0,1] \right)$ to be true, weaker regularity assumptions on $u_{0}$ would imply weaker conclusion in~\eqref{eq:Nstrong} (see \cite{MR1644183} for more details).

\subsection{Lie--Trotter time splitting}\label{sec:LTsplit}
The idea of a Lie--Trotter time splitting scheme for ordinary differential equations (ODEs) is to decompose an ODE of the form
\begin{equation*}
\left\lbrace
\begin{aligned}
& \frac{\diff y(t)}{\diff t} = f_{1}(y(t)) + f_{2}(y(t)),\ t \in (0,\Delta t], \\ 
& y(0) = y_{0},
\end{aligned}
\right.
\end{equation*}
that we wish to solve on a small interval $[0,\Delta t]$, into two ODEs
\begin{equation*}
\left\lbrace
\begin{aligned}
& \frac{\diff y_{1}(t)}{\diff t} = f_{1}(y_{1}(t)),\ t \in (0,\Delta t], \\ 
& y_{1}(0) = y_{0},
\end{aligned}
\right.
\end{equation*}
and
\begin{equation*}
\left\lbrace
\begin{aligned}
& \frac{\diff y_{2}(t)}{\diff t} = f_{2}(y_{2}(t)),\ t \in (0,\Delta t], \\ 
& y_{2}(0) = y_{1}(\Delta t),
\end{aligned}
\right.
\end{equation*}
that are solved in sequence. We refer the interested reader to \cite{MR2009376} for a complete exposition of splitting schemes.

\subsection{Exact simulation of SDEs}\label{sec:exactSim}
We dedicate this section to a summary of exact simulation for one-dimensional time-homogeneous SDEs of the form
\begin{equation}\label{eq:SDE-exactSim}
\left\lbrace
\begin{aligned}
& \diff X(t) = f(X(t)) \diff t + \lambda g(X(t)) \diff B(t),\ t \in (0,T], \\ 
& X(0) = x_{0} \in (a,b),
\end{aligned}
\right.
\end{equation}
for some $\lambda > 0$, where $f$ and $g$ satisfy the assumptions in Section~\ref{sec:spde}. Note that we in this section assume that $x_{0} \in (a,b)$. If $x_{0} \in \{a,b \}$, then the drift $f$ and the diffusion $g$ vanish (see Assumption~\ref{ass:g} and Assumption~\ref{ass:fg}) and $X(t) = x_{0}$ for all $t \in [0,T]$. In essence, the exact simulation algorithm developed in \cite{MR2274855,MR2187299} is an accept-reject sampling algorithm on sample paths. The exact simulation algorithm samples a proposal path of the SDE~\eqref{eq:SDE-exactSim} and then accepts or rejects based on a condition to preserve the distribution. The assumptions in Section~\ref{sec:spde} guarantees that, with probability one, the solution $X$ of~\eqref{eq:SDE-exactSim} with initial value $x_{0} \in (a,b)$ can only take values in the open domain $(a,b)$. We refer to \cite{MR4737060}, where the same assumptions on $f$ and $g$ are used, for more details on the boundary behaviour of the solution $X$ of~\eqref{eq:SDE-exactSim}. Relevant for this paper is $\lambda = \sqrt{N}$, where $N$ is the space discretisation parameter (see Section~\ref{sec:spacedisc}). The objective of this section is not to give a complete description of exact simulation, but rather to introduce enough for our purposes. See \cite{MR2274855,MR2187299} for a comprehensive exposition of exact simulation of diffusion processes. 

It turns out to be convenient to first apply a deterministic time-change to the SDE in~\eqref{eq:SDE-exactSim} to transfer the constant $\lambda>0$ in the diffusion coefficient to a factor $\lambda^{-2}$ in the drift coefficient. Most importantly, the setting used in exact simulation for SDEs~\cite{MR2187299} is easier to apply to the current setting after this deterministic time-change. For example, the Lamperti transform $\Phi$ in~\eqref{eq:LampTrans} below would be dependent on $\lambda$ without the deterministic time-change. See Appendix~\ref{sec:app-TC} for more details on this time-change. We introduce $\Theta (t) = \lambda^{-2} t$ and let $\widetilde{X}(t) = X(\Theta(t)) = X(\lambda^{-2} t)$. Then $\widetilde{X}$ satisfies
\begin{equation}\label{eq:SDE-exactSim-TC}
\left\lbrace
\begin{aligned}
& \diff \widetilde{X}(t) = \lambda^{-2} f(\widetilde{X}(t)) \diff t + g(\widetilde{X}(t)) \diff B(t),\ t \in (0,\lambda^{2} T], \\ 
& \widetilde{X}(0) = x_{0} \in (a,b),
\end{aligned}
\right.
\end{equation}
in the weak sense, and we recover the solution $X$ of the SDE in~\eqref{eq:SDE-exactSim} by $X(t) = \widetilde{X}(\Theta^{-1}(t)) = \widetilde{X}(\lambda^{2} t)$. Thus, sampling $X(\Delta t)$ given $x_{0}$ and is equivalent to sampling $\widetilde{X}(\lambda^{2} \Delta t)$ given $x_{0}$.

The Lamperti tranform $\Phi$ associated with the SDE in~\eqref{eq:SDE-exactSim-TC} is defined by
\begin{equation}\label{eq:LampTrans}
\Phi(r) = \int_{r_{0}}^{r} \frac{1}{g(w)} \diff w,\ r \in (a,b),
\end{equation}
for some choice of $r_{0} \in (a,b)$, and has the property that the stochastic process defined by $Y(t) = \Phi(\widetilde{X}(t))$ satisfies the following SDE with additive noise
\begin{equation}\label{eq:SDEexactLamp}
\left\lbrace
\begin{aligned}
& \diff Y(t) = \alpha(Y(t)) \diff t + \diff B(t),\ t \in (0,\lambda^{2} T], \\ 
& Y(0) = \Phi(x_{0}) \in \mathbb{R},
\end{aligned}
\right.
\end{equation}
where the drift coefficient $\alpha$ is given by
\begin{equation*}
\alpha(r) = \frac{1}{\lambda^{2}} \frac{f(\Phi^{-1}(r))}{g(\Phi^{-1}(r))} - \frac{1}{2} g'(\Phi^{-1}(r)),\ r \in \mathbb{R}.
\end{equation*}
Recall that the function $\Phi^{-1}: \mathbb{R} \to (a,b)$ is bounded, bijective, continuously differentiable and has bounded derivative by Proposition~\ref{prop:phiinv}. Observe that the transformation $Y(t) = \Phi(\widetilde{X}(t))$ is well-defined if $\widetilde{X}$ stays in the open domain $(a,b)$ and observe that if we let $r_{0} = x_{0}$ then $Y(0) = 0$. In fact, we let $r_{0} = x_{0}$ to comply with \cite{MR2187299}. Let $T_{\lambda} = \lambda^{2} T$ for ease of presentation.

We are now in a position to describe the exact simulation and we do so for the SDE in~\eqref{eq:SDE-exactSim-TC} for convenience and as this resembles the exposition in \cite{MR2187299} the most. Here we provide an overview of the idea behind exact simulation, in Appendix~\ref{app:implementation} we provide some extra implementation details for the numerical experiments in Section~\ref{sec:num} and we refer to \cite{MR2187299} for the complete picture. The exact simulation algorithm can be briefly summarised as follows
\begin{enumerate}[(i)]
\item Use the Lamperti transform $\Phi$ defined in~\eqref{eq:LampTrans} to convert the SDE in~\eqref{eq:SDE-exactSim-TC} into the SDE in~\eqref{eq:SDEexactLamp} with drift $\alpha$ and additive noise.
\item Sample a skeleton $\omega_{1}, \ldots, \omega_{L}$ (at certain random time points $\tau_{1},\ldots,\tau_{L}$) from the distribution of a biased Brownian motion with end point distribution $h$, where $\omega_{L} \sim h$ and $\tau_{L} = T_{\lambda}$ is the (deterministic) end point of the interval under consideration and where
\begin{equation}\label{eq:hdef}
h(r) \propto \exp \left( A(r) - \frac{r^{2}}{2 T_{\lambda}} \right),\ r \in \mathbb{R},
\end{equation}
with
\begin{equation}\label{eq:Adef}
A(r) = \int_{0}^{r} \alpha(y) \diff y,\ r \in \mathbb{R}.
\end{equation}
Note that $\propto$ in~\eqref{eq:hdef} means that we have discarded some factors not depending on $r \in \mathbb{R}$.
\item Compute a reject probability based on $\omega_{1},\ldots,\omega_{L}$ to decide whether or not to reject $\omega_{1},\ldots,\omega_{L}$. If rejected, then go back to (ii). If required or desired, sample additional point(s) $\widetilde{\omega}$ at any (non-random) time point(s) $\widetilde{\tau} \in (\tau_{k},\tau_{k+1})$ using a Brownian bridge(s) between $\omega_{k}$ and $\omega_{k+1}$.
\item Apply the inverse Lamperti transform $\Phi^{-1}$ to the skeleton (re-enumerated to include additional points) $\omega_{1},\ldots,\omega_{L}$.
\end{enumerate}
The above described algorithm produces sampled points $\Phi^{-1}(\omega_{0}),\ldots,\Phi^{-1}(\omega_{L})$ corresponding to the time points $\tau_{0},\ldots,\tau_{L}$ that have the same distribution as $\widetilde{X}(\tau_{0}),\ldots,\widetilde{X}(\tau_{L})$. To obtain sampled points from the original SDE~\eqref{eq:SDE-exactSim} we apply the inverse time-change $\Theta^{-1}$. At this point, we would like to raise three points. Firstly, in order to use the exact simulation algorithm we have to be able to exactly sample from the above defined distribution function $h$, and $h$ depends on the choices of $f$ and $g$. If $h$ is not a standard distribution that is easily sampled from, then we have to resort to other means (e.g., standard accept-reject sampling). We refer to the classical book \cite{MR836973} for many approaches to sampling from distributions. Secondly, the fact that step (iii) can be exactly evaluated in a computer is nontrivial, and we refer the interested reader to \cite{MR2187299}. Thirdly, the above construction is well-defined and executable under the assumptions in Section~\ref{sec:spde}.

(i) and (iv) are self-explanatory, it remains to consider (ii) and (iii) for the transformed SDE with additive noise in~\eqref{eq:SDEexactLamp}. For (ii), given the random time points $\tau_{1},\ldots,\tau_{L-1}$ and $\omega_{L} \sim h$ at $\tau_{L} = T_{\lambda}$, we sample $\omega_{1},\ldots,\omega_{L-1}$ using Brownian bridges. The random time points $\tau_{1},\ldots,\tau_{L-1}$ are chosen so that the accept-reject probability in (iii) can be computed from only finitely many values of $\left( \omega(t) \right)_{t \in [0,T_{\lambda}]}$. More details on this will be included below.

We now derive the accept-reject probability for (iii), this will also motivate the precise construction in (ii). First we compare the sample paths $\omega \in \C([0,T_{\lambda}])$ of the SDE for standard Brownian motion under the standard probability measure $\PP$ (on $\Omega$)
\begin{equation}\label{eq:exact-BM}
\left\lbrace
\begin{aligned}
& \diff Y(t) = \diff B(t),\ t \in (0, T_{\lambda}], \\ 
& Y(0) = 0,
\end{aligned}
\right.
\end{equation}
to the sample paths of the SDE with additive noise
\begin{equation}\label{eq:exact-ad}
\left\lbrace
\begin{aligned}
& \diff Y(t) = \alpha(Y(t)) \diff t + \diff B(t),\ t \in (0,T_{\lambda}], \\ 
& Y(0) = 0,
\end{aligned}
\right.
\end{equation}
using Girsanov's theorem. To this end, we let $\WW$ be the probability measure on $\C([0,T_{\lambda}])$ induced by~\eqref{eq:exact-BM}, $\QQ$ be the probability measure on $\C([0,T_{\lambda}])$ induced by~\eqref{eq:exact-ad} and
\begin{equation}\label{eq:Mtexpr1}
M_{t}(\omega) = \EXP \left( \int_{0}^{t} \alpha \left( \omega(s) \right) \diff B(s) - \frac{1}{2} \int_{0}^{t} \alpha^{2} \left( \omega(s) \right) \diff s \right),\ t \in [0,T_{\lambda}],\ \omega \in \C([0,T_{\lambda}]).
\end{equation}
The Radon-Nikodym derivative is then given by $\frac{\diff \QQ}{\diff \WW} = M_{T_{\lambda}}$. This follows from by Girsanov's theorem (see, e.g., Theorem $8.6.4$ in \cite{MR2001996}), provided that $\left( M_{t} \right)_{t \in [0,T_{\lambda}]}$ is a martingale under the probability measure $\WW$. Indeed, $\left( M_{t} \right)_{t \in [0,T_{\lambda}]}$ is a martingale under $\WW$, by, e.g., Novikov's condition (see, e.g., equation $8.6.8$ in \cite{MR2001996}) and Proposition~\ref{prop:alpha}. The natural approach would be to construct an accept-reject probability based on the Radon-Nikodym derivative $\frac{\diff \QQ}{\diff \WW}(\omega) = M_{T_{\lambda}}(\omega)$, for $\omega \in \C([0,T_{\lambda}])$. The problem with this approach is that we, in that case, would have to evaluate~\eqref{eq:Mtexpr1} (for $t = T_{\lambda}$) exactly, and this is not possible by the presence of the stochastic integral. 

By recalling the definition of $A$ in equation~\eqref{eq:Adef} and by using Ito's formula, we can avoid the need to compute the stochastic integral in~\eqref{eq:Mtexpr1} by rewrite $M_{t}$ as follows
\begin{equation}\label{eq:Mt}
M_{t}(\omega) = \exp \left( A(\omega(t) - A(\omega(0)) \right) \times \exp \left( - \frac{1}{2} \int_{0}^{t} \alpha'(\omega(s)) + \alpha^{2}(\omega(s)) \diff s \right),\ t \in [0,T_{\lambda}].
\end{equation}
Recall that $\alpha$ is bounded and twice differentiable with bounded first and second derivatives by Proposition~\ref{prop:alpha}. The first factor involving $A(\omega(t))$ can be unbounded and is hence problematic. We therefore, instead, consider the Random-Nikodym derivative $\frac{\diff \QQ}{\diff \Z} = \frac{\diff \QQ}{\diff \WW} \frac{\diff \WW}{\diff \Z}$ where $\Z$ is the probability measure on $\C([0,T_{\lambda}])$ induced by Brownian motion conditioned on $B(T_{\lambda}) \sim h$ where $h$ is the distribution function given in equation~\eqref{eq:hdef}. One can then show that (see, e.g., Proposition $3$ in \cite{MR2187299})
\begin{equation*}
\frac{\diff \Z}{\diff \WW}(\omega) = \frac{\sqrt{2 \pi T_{\lambda}} h(\omega(T_{\lambda}))}{ \exp \left( - \frac{(\omega(T_{\lambda})-\omega(0))^{2}}{2 T_{\lambda}} \right)}
\end{equation*}
and in turn that
\begin{equation}\label{eq:QZderiv}
\frac{\diff \QQ}{\diff \Z}(\omega) \propto \exp \left( - \frac{1}{2} \int_{0}^{T_{\lambda}} \alpha'(\omega(s)) + \alpha^{2}(\omega(s)) \diff s \right),
\end{equation}
where $\propto$ means that we have discarded some factors not depending on $\omega$. Notice that~\eqref{eq:QZderiv} is similar to the expression in~\eqref{eq:Mt} except that the possibly unbounded factor of $A(Y(t))$ is no longer present, and there is hope that we can construct an accept-reject probability based on~\eqref{eq:QZderiv}. With the aim to normalise the right hand side of~\eqref{eq:QZderiv}, we let $k_{1},k_{2} \in \mathbb{R}$ be such that
\begin{equation}\label{eq:phiB}
k_{1} \leq \frac{1}{2} \alpha'(r) + \frac{1}{2} \alpha^{2}(r) \leq k_{2},
\end{equation}
for every $t \in \mathbb{R}$, which is possible by Proposition~\ref{prop:alpha}. We remark that the assumption in equation~\eqref{eq:phiB} can be relaxed without hindering us from using the exact simulation algorithm. See, for example, Section $3$ in \cite{MR2274855} for more details. By introducing
\begin{equation*}
\phi(r) = \frac{1}{2} \alpha'(r) + \frac{1}{2} \alpha^{2}(r) - k_{1},\ r \in \mathbb{R},
\end{equation*}
we can write~\eqref{eq:QZderiv} on the more compact form (see also, e.g., equation $5$ in \cite{MR2187299})
\begin{equation*}
\frac{\diff \QQ}{\diff \Z}(\omega) \propto \exp \left( - \int_{0}^{T_{\lambda}} \phi(\omega(s)) \diff s \right) \leq 1.
\end{equation*}
We accept or reject an element $\omega \in \C([0,T_{\lambda}])$ as follows. Suppose that $\omega \sim \Z$ is a sample path from the Brownian motion conditioned on $B(T_{\lambda}) \sim h$. We accept the path $\omega$ with probability $p(\omega) = \exp \left( - \int_{0}^{T_{\lambda}} \phi(\omega(s)) \diff s \right)$ and we  reject the path $\omega$ with probability $1-p(\omega)$. If we reject $\omega$, then we repeat the procedure with a new sample $\omega$ from probability measure $\Z$. In practice, we cannot evaluate $p(\omega)$ as it depends on infinitely many values $\left( \omega(t) \right)_{t \in [0,T_{\lambda}]}$. To circumvent this issue, let $\mathcal{N}$ denote the number of points below the graph $\{ (t,\phi(\omega(t))), t \in [0,T_{\lambda}] \}$ of a homogeneous Poisson process with unit intensity on $[0,T_{\lambda}] \times [k_{2}-k_{1}]$. Then, as proved and stated in a clear way in Theorem $1$ in \cite{MR2274855}, we have that
\begin{equation*}
\PP \left( \mathcal{N} = 0 \ | \ \omega \right) = \exp \left( - \int_{0}^{T_{\lambda}} \phi(\omega(s)) \diff s \right) = p(\omega);
\end{equation*}
that is, we can compute $p(\omega)$ from finitely many values of $\omega$ at random time points.
 
Lastly, to guarantee that the probability of accepting a sample path $\omega \in \C([0,T_{\lambda}])$ is bounded from below by $e^{-1}$, we assume that the end time point $T_{\lambda}$ is bounded by
\begin{equation}\label{eq:TLbound}
T_{\lambda} \leq \frac{1}{k_{2} - k_{1}}.
\end{equation}
See Proposition $4$ in \cite{MR2187299} or Proposition $3$ in \cite{MR2274855} for more details. The condition~\eqref{eq:TLbound} is satisfies for each numerical example in Section~\ref{sec:num} when using the coupling $\Delta t = \Delta x^{2}$. This completes the description of the exact simulation algorithm.

\section{The boundary-preserving scheme}\label{sec:scheme}
In this core part of the paper, we give a complete description of the boundary-preserving scheme and state and prove that the scheme is boundary-preserving and sharp weak convergence order. First, we let $M \in \mathbb{N}$ and introduce the time grid size $\Delta t = T/M$ and the corresponding time grid points $t_{m} = m \Delta t$, for $m = 0, \ldots,M$. Both $N$ and $\Delta x$ will be used for notational convenience. Associated with the introduced time grid, we let $\ell^{M}:[0,T] \to \{ t_{0},\ldots, t_{M} \}$ be the step function defined by $\ell^{M}(t) = t_{m}$ for $t \in [t_{m},t_{m+1})$ for $m=0,\ldots,M-1$ and $\ell^{M}(T)=T$. We first provide the definition of the scheme in Section~\ref{sec:schemeconstruction}. We then state and provide the main results in Section~\ref{sec:results}.

\subsection{Description of the integrator}\label{sec:schemeconstruction}
The idea of the proposed \textbf{L}ie--\textbf{T}rotter with \textbf{E}xact simulation and integration ($\LTE$) scheme is to first use a spatial finite difference discretisation and then for the temporal discretisation to apply a Lie--Trotter time splitting followed by exact simulation and exact integration, respectively, of the two resulting subproblems. More precisely, first we apply a standard finite difference discretisation in space (see Section~\ref{sec:spacedisc}) to the SPDE in equation~\eqref{intro:SPDE} to obtain the following system of stochastic ordinary differential equations (SDEs)
\begin{equation}\label{eq:FDeq}
\diff u^{N}(t) = \left( N^{2} D^{N} u^{N}(t) + f(u^{N}(t)) \right) \diff t + \sqrt{N} g(u^{N}(t)) \diff W^{N}(t).
\end{equation}
The second step is to split the above system of SDEs in~\eqref{eq:FDeq} into one diagonal system of SDEs given by
\begin{equation}\label{eq:split1}
\diff v^{M,N,1}(t) = f(v^{M,N,1}(t)) \diff t + \sqrt{N} g(v^{M,N,1}(t)) \diff W^{N}(t),
\end{equation}
and one non-diagonal system of ordinary differential equations (ODEs) given by
\begin{equation}\label{eq:split2}
\diff v^{M,N,2}(t) = N^{2} D^{N} v^{M,N,2}(t) \diff t.
\end{equation}
By the assumptions in Section~\ref{sec:spde}, the system~\eqref{eq:split1} can be exactly simulated (see Section~\ref{sec:exactSim}), and the system~\eqref{eq:split2} can be exactly integrated using the exponential map $e^{t N^{2} D^{N}}$. Recall that we exclude the first component $n=0$ and the last component $n = N$ by the Dirichlet boundary conditions. The third step is to, starting at time grid point $t_{m}$, sample according to the dynamics of equation~\eqref{eq:split1} and integrate equation~\eqref{eq:split2} in succession on a small time interval $[t_{m},t_{m}+\Delta t]$. More precisely, given the numerical approximation $u^{\LTE}_{m} = \left( u^{\LTE}_{m,n} \right)_{1 \leq n \leq N-1}$ at some time grid point $t_{m} = m \Delta t$, where $0 \leq m \leq M-1$, we compute the numerical approximation $u^{\LTE}_{m+1}$ at the next time grid point $t_{m+1} = t_{m} + \Delta t$ as follows:
\begin{enumerate}
\item Sample according to the dynamics of the diagonal system of SDEs
\begin{equation}\label{eq:subprob1}
\left\lbrace
\begin{aligned}
& \diff v^{M,N,1}_{m}(t) = f(v^{M,N,1}_{m}(t)) \diff t + \sqrt{N} g(v^{M,N,1}_{m}(t)) \diff W^{N}(t), \\ 
& v^{M,N,1}_{m}(t_{m}) = u^{\LTE}_{m},
\end{aligned}
\right.
\end{equation}
at time $t_{m+1}$ to obtain the intermediate state $u^{\LTE}_{m+1/2} = v^{\LTE}_{m+1} = v^{M,N,1}_{m}(t_{m+1})$. We sample according to~\eqref{eq:subprob1} by applying exact simulation componentwise, see Section~\ref{sec:exactSim} for details.
\item Integrate the tridiagonal linear system of ODEs
\begin{equation}\label{eq:subprob2}
\left\lbrace
\begin{aligned}
& \diff v^{M,N,2}_{m}(t) = N^{2} D^{N} v^{M,N,2}_{m}(t) \diff t, \\ 
& v^{M,N,2}_{m}(t_{m}) = u^{\LTE}_{m+1/2} = v^{\LTE}_{m+1},
\end{aligned}
\right.
\end{equation}
on $[t_{m},t_{m+1}]$ to obtain the numerical approximation
\begin{equation}\label{eq:solsubprob2}
u^{\LTE}_{m+1} = v^{M,N,2}_{m}(t_{m+1}) = e^{\Delta t N^{2} D^{N}} v_{m}^{M,N,1}(t_{m+1})
\end{equation}
at the next time grid point $t_{m+1}$.
\end{enumerate}

A key property of the $\LTE$ scheme is that it is guaranteed to remain in the invariant domain $\CD$. The following lemma will be used to prove that the LTE scheme is boundary-preserving.
\begin{lemma}\label{lem:heatKBP}
If $v \in [a,b]^{N-1}$, then $e^{t N^{2} D^{N}} v \in [a,b]^{N-1}$ for every $t \in [0,T]$.
\end{lemma}
\begin{proof}
It suffices to show the case $a=0$, since we can decompose $v = \max(v,0) - \max(-v,0)$, into its positive and negatives parts, and use the case $a=0$ for each part separately. We set $v_{0}=v_{N}=0$ and extend $v$ to a function $\bar{v}:[0,1] \to [0,b]$ by linear interpolation. By positivity of the heat kernel and by the estimate in equation $(3.5)$ in \cite{MR1644183}, we have that
\begin{equation*}
\int_{0}^{1} G^{N}(t,x_{n},y) \bar{v}(\kappa^{N}(y)) \diff y \leq \max_{x \in [0,1]} |\bar{v}(x)| = \max_{k=0,\ldots,N} |v_{k}| \leq b.
\end{equation*}
We evaluate the integral
\begin{equation*}
  \begin{split}
    \int_{0}^{1} G^{N}(t,x_{n},y) \bar{v}(\kappa^{N}(y)) \diff y &= \sum_{k=1}^{N-1} \Delta x G^{N}(t,x_{n},x_{k}) \bar{v}(x_{k}) = \sum_{k=1}^{N-1} G^{N}_{nk}(t) v_{k} \\ &= \left( e^{t N^{2} D^{N}} v \right)_{n} \geq 0,
  \end{split}
\end{equation*}
by assumption, to conclude that $e^{t N^{2} D^{N}} v \in [0,b]^{N-1}$.
\end{proof}

\begin{proposition}\label{prop:LTE-BP}
Let $M, N \in \mathbb{N}$ and suppose the assumptions in Section~\ref{sec:spde} are satisfied. Let $u^{\LTE}_{m,n}$, for $0 \leq m \leq M$ and for $0 \leq n \leq N$, be constructed by the $\LTE$ scheme defined above with initial value $u^{\LTE}_{0,n} = u_{0}(x_{n})$. Then
\begin{equation*}
\PP \left( u^{\LTE}_{m,n} \in \CD = [a,b],\ \forall m=0,\ldots,M,\ \forall n=0,\ldots,N \right) = 1.
\end{equation*}
\end{proposition}
\begin{proof}[Proof of Proposition~\ref{prop:LTE-BP}]
We use induction on the time index $m = 0,\ldots, M$ to obtain the result.
\begin{enumerate}
\item By Assumption~\ref{ass:u0}, the statement is true for the base case: $u^{\LTE}_{0,n} = u(0,x_{n}) \in \CD$ for every $n = 0,\ldots,N$.
\item Suppose that $u^{\LTE}_{m,n} \in \CD$, for every $n=0,\ldots,N$, for some $m=0,\ldots, M-1$. It suffices to show that both sub-problems in equations~\eqref{eq:subprob1} and~\eqref{eq:subprob2} preserve this property. By componentwise applying Feller's boundary classification (see, e.g., \cite{Karlin1981ASC}) or the author's related work \cite{MR4737060}, it follows that the intermediate state $u^{\LTE}_{m+1/2,n} = v^{\LTE}_{m+1,n} = \left( v^{\LTE}_{m+1} \right)_{n}$ from the first sub-problem~\eqref{eq:subprob1} satisfies $u^{\LTE}_{m+1/2,n} \in \CD$ for every $n=1,\ldots,N-1$. By Lemma~\ref{lem:heatKBP}, we conclude that the solution to the second sub-problem~\eqref{eq:subprob2} satisfies $u^{\LTE}_{m+1,n} \in \CD$ for every $n = 1,\ldots,N-1$. 
\end{enumerate}
\end{proof}
\subsubsection{Auxiliary space-time continuous extension}
We have so far constructed the $\LTE$ approximation on the discretised space-time grid and we now describe how to extend this numerical scheme to be defined for all $(t,x) \in [0,T] \times [0,1]$. The extended scheme plays a crucial role in the convergence proof. More precisely, it allows us to compare the mild integral equations for $u^{N}$ in~\eqref{eq:Nmild} to a similar mild integral representation of $u^{\LTE}$. The space-time extension follows the same ideas as the space-time extension used by the author and collaborators in \cite{MR4780408}. 

As exact simulation is an accept-reject algorithm with re-sampling the noise path and we use it to treat sub-problem $1$ in~\eqref{eq:subprob1}, we have to use a different driving noise for the mild integral equation for the space-time extension of $u^{\LTE}$. More precisely, let $\widetilde{W}^{N}(t)$ denote the $(N-1)$-dimensional Brownian motion that corresponds to the exact simulation in sub-problem $1$ in~\eqref{eq:subprob1} and let $\widetilde{W}$ denote a space-time white noise with the property that $\widetilde{W}^{N}_{n}(t) = \sqrt{N} \left( \widetilde{W}(t,x_{n+1}) - \widetilde{W}(t,x_{n}) \right)$.

We first extend in time and then we extend in space. For $m \in \{ 0,\ldots,M-1 \}$, we let
\begin{equation*}
u^{\LTE}_{m}(t) = e^{(\ell^{M}(t) - t_{m}) N^{2} D^{N}} v_{m}^{M,N,1}(t)
\end{equation*}
for $t \in [t_{m},t_{m+1}]$, where $v_{m}^{M,N,1}$ is defined in~\eqref{eq:subprob1}. Note that $u^{\LTE}_{m}(t)$ is not the same object as $u^{\LTE}_{m}$. Then the following properties hold:
\begin{enumerate}
\item if $t \in [t_{m},t_{m+1})$, then $u^{\LTE}_{m}(t) = v_{m}^{M,N,1}(t)$.
\item $u^{\LTE}_{m}(t_{m}) = v_{m}^{M,N,1}(t_{m})=u^{\LTE}_{m}$.
\item $u^{\LTE}_{m}(t_{m+1}) = e^{\Delta t N^{2} D^{N}} v_{m}^{M,N,1}(t_{m+1}) = u^{\LTE}_{m+1}$.
\end{enumerate}
Let now $u_{n}^{\LTE}(t) = \left( u^{\LTE}_{m}(t) \right)_{n}$, for every $n = 1,\ldots,N-1$ and every $t \in [0,T]$. Note that, as $u^{\LTE}_{m}(t_{m}) = u^{\LTE}_{m}$ and $u^{\LTE}_{m}(t_{m+1}) = u^{\LTE}_{m+1}$, we may drop the dependence on $m \in \{ 0,\ldots,M \}$ without any confusion. We now claim that
\begin{equation}\label{eq:uLTEnMild}
  \begin{split}
    u_{n}^{\LTE}(t) &= \sum_{k=1}^{N-1} G_{nk}^{N}(\ell^{M}(t)) u_{0}(x_{k}) + \int_{0}^{t} \sum_{k=1}^{N-1} G_{nk}^{N}(\ell^{M}(t)-\ell^{M}(s)) f(u_{k}^{\LTE}(s)) \diff s \\ &+ \sqrt{N} \int_{0}^{t} \sum_{k=1}^{N-1} G_{nk}^{N}(\ell^{M}(t)-\ell^{M}(s)) g(u^{\LTE}_{k}(s)) \diff \widetilde{W}^{N}_{k}(s),
  \end{split}
\end{equation}
for every $t \in [0,T]$. Indeed, let $v_{m,k}^{M,N,1}(t) = \left( v_{m}^{M,N,1}(t) \right)_{k}$ for every $k = 1, \ldots, N-1$, then for $t \in [t_{m},t_{m+1}]$ we have
\begin{align*}
u_{n}^{\LTE}(t) &= \left( u^{\LTE}_{m}(t) \right)_{n} = \left( e^{(\ell^{M}(t) - t_{m}) N^{2} D^{N}} v_{m}^{M,N,1}(t) \right)_{n} = \sum_{k=1}^{N-1} G_{nk}^{N}(\ell^{M}(t)-t_{m}) v_{m,k}^{M,N,1}(t) \\ &= \sum_{k=1}^{N-1} G_{nk}^{N}(\ell^{M}(t)-t_{m}) u_{m,k}^{\LTE}(t_{m}) + \int_{t_{m}}^{t} \sum_{k=1}^{N-1} G_{nk}^{N}(\ell^{M}(t)-t_{m}) f(u^{\LTE}_{m,k}(s)) \diff s \\ &+ \sqrt{N} \int_{t_{m}}^{t} \sum_{k=1}^{N-1} G_{nk}^{N}(\ell^{M}(t)-t_{m}) g(u^{\LTE}_{m,k}(s)) \diff \widetilde{W}_{k}^{N}(s),
\end{align*}
where we used that $v_{m}^{M,N,1}(t)$ solves the diagonal system of SDEs in~\eqref{eq:subprob1} with initial value $u_{m}^{\LTE} = u^{\LTE}_{m}(t_{m})$ and that $v_{m,k}^{M,N,1}(s) = u_{m,k}^{\LTE}(s)$ for $s \in [t_{m},t_{m+1})$. Repeating this argument and using that $u^{\LTE}_{m,k}(s) = u^{\LTE}_{k}(s)$, for $s \in [t_{m},t_{m+1})$, gives the desired expression in~\eqref{eq:uLTEnMild}. Note that, by Feller's boundary classificiation, the extension in time and Proposition~\ref{prop:LTE-BP} implies that $u^{\LTE}_{n}$ only takes values in $\CD = [a,b]$ almost surely
\begin{equation*}
\PP \left( u^{\LTE}_{n}(t) \in \CD = [a,b],\ \forall t \in [0,T],\ \forall n=1,\ldots,N-1 \right) = 1.
\end{equation*}

We now describe the extension in the spatial variable. We define $u^{\LTE}$ as the solution to the following integral equation
\begin{equation}\label{eq:LTEmild}
  \begin{split}
    u^{\LTE}(t,x)& = \int_{0}^{1} G^{N}(\ell^{M}(t),x,y) u_{0}(\kappa^{N}(y)) dy \\ &+ \int_{0}^{t} \int_{0}^{1} G^{N}(\ell^{M}(t)-\ell^{M}(s),x,y) f(u^{\LTE}(s,\kappa^{N}(y))) ds dy \\ &+ \int_{0}^{t} \int_{0}^{1} G^{N}(\ell^{M}(t)-\ell^{M}(s),x,y) g(u^{\LTE}(s,\kappa^{N}(y))) \diff \widetilde{W}(s,y)
  \end{split}
\end{equation}
for every $(t,x) \in [0,T] \times [0,1]$, almost surely. Notice that~\eqref{eq:LTEmild} is similar to the mild equation~\eqref{eq:Nmild} for the semi-discrete approximation $u^{N}$. As will be shown later on, this similarity will be important for the convergence proof. The following lemma gives us that $u^{\LTE}$ as defined in~\eqref{eq:LTEmild} is a spatial extension of $u^{\LTE}_{n}$. In fact, it follows from~\eqref{eq:LTEmild} that $u^{\LTE}$ is the piecewise linear extension of $u^{\LTE}_{n}$.
\begin{lemma}\label{lem:LTEmild}
Let $m \in \{0,\ldots,M \}$ and $n \in \{ 0,\ldots,N \}$. Then $u^{\LTE}(t_{m},x_{n}) = u^{\LTE}_{n}(t_{m}) = u^{\LTE}_{m,n}$, almost surely.
\end{lemma}
\begin{proof}[Proof of Lemma~\ref{lem:LTEmild}]
By using the property $N G_{nk}^{N}(t) = G^{N}(t,x_{n},x_{k})$, we insert  $x=x_{n}$ in~\eqref{eq:LTEmild} to see that
\begin{align*}
u^{\LTE}(t,x_{n}) &= \sum_{k=1}^{N-1} \int_{x_{k}}^{x_{k+1}} G^{N}(\ell^{M}(t),x_{n},x_{k}) u_{0}(x_{k}) \diff y \\ &+ \int_{0}^{t} \sum_{k=1}^{N-1} \int_{x_{k}}^{x_{k+1}} G^{N}(\ell^{M}(t)-\ell^{M}(s),x_{n},x_{k}) f(u^{\LTE}(s,x_{k})) \diff y \diff s \\ &+ \int_{0}^{t} \sum_{k=1}^{N-1} \int_{x_{k}}^{x_{k+1}} G^{N}(\ell^{M}(t)-\ell^{M}(s),x_{n},x_{k}) g(u^{\LTE}(s,x_{k})) \diff \widetilde{W}(s,y)\\ &= \sum_{k=1}^{N-1}G^{N}_{nk}(\ell^{M}(t)) u_{0}(x_{k}) + \int_{0}^{t} \sum_{k=1}^{N-1} G^{N}_{nk}(\ell^{M}(t)-\ell^{M}(s)) f(u^{\LTE}(s,x_{k})) \diff s \\ &+ \sqrt{N} \int_{0}^{t} \sum_{k=1}^{N-1} G^{N}_{nk}(\ell^{M}(t)-\ell^{M}(s)) g(u^{\LTE}(s,x_{k})) \diff \widetilde{W}_{k}^{N}(s),
\end{align*}
where we decomposed the integrals as $\int_{0}^{1} = \sum_{k=1}^{N-1} \int_{x_{k}}^{x_{k+1}}$ and 
where we for the third term on the right hand side also used that $\int_{x_{k}}^{x_{k+1}} \diff \widetilde{W}(s,y) = \diff \widetilde{W}_{k}^{N}(s)$ by definition. Thus, $u^{\LTE}(t,x_{n})$ and $u^{\LTE}_{n}(t)$ satisfy the same equation and must therefore be equal. We obtain the desired property by inserting $t = t_{m}$
\begin{equation*}
u^{\LTE}(t_{m},x_{n}) = u^{\LTE}_{n}(t_{m}) = u^{\LTE}_{m,n}.
\end{equation*}
\end{proof}
Note that by Proposition~\ref{prop:LTE-BP} and Lemma~\ref{lem:LTEmild}, we obtain boundary preservation of the full space-time extension $u^{\LTE}$ satisfies
\begin{equation*}
\PP \left( u^{\LTE}(t,x) \in \CD = [a,b],\ \forall t \in [0,T],\ \forall x \in [0,1] \right) = 1.
\end{equation*}
Recall that $u^{LTE}$ is the piecewise linear extension of $u^{LTE}_{n}$ by~\eqref{eq:LTEmild}.

\begin{lemma}\label{lem:uLTEtimeReg}
Let $T>0$ and let $t \in [0,T]$. We have the following time increment estimate for the $\LTE$ scheme
\begin{equation*}
\sup_{n=0,\ldots,N} \E \left[ \left| u^{\LTE}(t,x_{n}) - u^{\LTE}(\ell^{M}(t),x_{n}) \right|^2 \right] \leq C(f,g) \left( \Delta t^{2} + \frac{\Delta t}{\Delta x} \right).
\end{equation*}
\end{lemma}
\begin{proof}
Recall that $u^{\LTE}(t,x_{n})$ satisfies
\begin{equation*}
u^{\LTE}(t,x_{n}) = u^{\LTE}(\ell^{M}(t),x_{n}) + \int_{\ell^{M}(t)}^{t} f(u^{\LTE}(s,x_{n})) \diff s + \sqrt{N} \int_{\ell^{M}(t)}^{t} g(u^{\LTE}(s,x_{n})) \diff \widetilde{W}^{N}_{n}(s),
\end{equation*}
since $u^{\LTE}(t,x_{n})$ solves the SDE in~\eqref{eq:subprob1} on $[\ell^{M}(t),\ell^{M}(t)+\Delta t)$ starting from $\ell^{M}(t)$ with initial value $u^{\LTE}(\ell^{M}(t),x_{n})$. Thus,
\begin{align*}
\E \left[ \left| u^{\LTE}(t,x_{n}) - u^{\LTE}(\ell^{M}(t),x_{n}) \right|^{2} \right] &\leq C \left( \E \left[ \left| \int_{\ell^{M}(t)}^{t} f(u^{\LTE}(s,x_{n})) \diff s \right|^{2} \right] \right. \\ &+ \left.  N \E \left[ \left| \int_{\ell^{M}(t)}^{t} g(u^{\LTE}(s,x_{n})) \diff \widetilde{W}^{N}_{n}(s) \right|^{2} \right] \right).
\end{align*}
We bound the first term using Jensen's inequality for integrals
\begin{equation*}
\E \left[ \left| \int_{\ell^{M}(t)}^{t} f(u^{\LTE}(s,x_{n})) \diff s \right|^{2} \right] \leq C(f) \Delta t^{2}
\end{equation*}
and we bound the second using the Itô isometry
\begin{equation*}
\E \left[ \left| \int_{\ell^{M}(t)}^{t} g(u^{\LTE}(s,x_{n})) \diff \widetilde{W}^{N}_{n}(s) \right|^{2} \right] \leq C(g) \Delta t,
\end{equation*}
to obtain the desired estimate
\begin{equation*}
\E \left[ \left| u^{\LTE}(t,x_{n}) - u^{\LTE}(\ell^{M}(t),x_{n}) \right|^{2} \right] \leq C(f,g) \left( \Delta t^{2} + \frac{\Delta t}{\Delta x} \right).
\end{equation*}
Note that we also used that $u^{\LTE}(t,x) \in \CD = [a,b]$ for all $t \in [0,T]$ and for all $x \in [0,1]$ and that $f$ and $g$ satisfy Assumption~\ref{ass:f} and Assumption~\ref{ass:g}, respectively.
\end{proof}

\subsection{Convergence result}\label{sec:results}
In the following, we prove that the proposed $\LTE$ scheme converges weakly to the exact solution of the considered SPDE in~\eqref{intro:SPDE}. This is the content of the following theorem.
\begin{theorem}\label{th:mainWeak}
Let the assumptions in Section~\ref{sec:spde} be satisfied. Let $u^{\LTE}$ be the $\LTE$ approximation defined in Section~\ref{sec:scheme} and let $u$ be the exact solution to the SPDE in equation~\eqref{intro:SPDE}. Then for any globally Lipschitz continuous $F: \mathbb{R} \to \mathbb{R}$ there exists a constant $C(F,T,f,g)$ such that
\begin{equation}
\sup_{m=0,\ldots,M} \sup_{n=0,\ldots,N} \left| \E \left[ F \left( u^{\LTE}_{m,n} \right) \right] - \E \left[ F \left( u(t_{m},x_{n}) \right) \right] \right| \leq C(F,T,f,g) \left( \Delta x^{1/2} + \sqrt{\Delta t^{2} + \frac{\Delta t}{\Delta x}} \right).
\end{equation}
If $\Delta t = \Delta x^{2}$, then
\begin{equation}
\sup_{m=0,\ldots,M} \sup_{n=0,\ldots,N} \left| \E \left[ F \left( u^{\LTE}_{m,n} \right) \right] - \E \left[ F \left( u(t_{m},x_{n}) \right) \right] \right| \leq C(F,T,f,g) \Delta t^{1/4}.
\end{equation}
\end{theorem}
As we use exact simulation as part of the proposed LTE scheme, we cannot expect to obtain strong convergence in Theorem~\ref{th:mainWeak}. We also remark that the order of convergence in Theorem~\ref{th:mainWeak} is what is expected for strong convergence. In other words, for globally Lipschitz continuous test functions $F$, the order of weak convergence is equal to the expected strong convergence rate. This is in agreement with, for example, \cite{MR4032336}. The proof of Theorem~\ref{th:mainWeak} follows from Theorem $3.1$ in \cite{MR1644183} and from the following proposition where we compare the $\LTE$ approximation $u^{\LTE}$ to the finite difference approximation $u^{N}$ with respect to a different driving noise than used in Section~\ref{sec:spacedisc}. 
\begin{proposition}\label{prop:mainStrong}
Let the assumptions in Section~\ref{sec:spde} be satisfied. Let $u^{\LTE}$ be the $\LTE$ approximation defined in Section~\ref{sec:scheme} and let $u^{N}$ be finite difference solution defined as the solution to~\eqref{eq:Nmild} with respect to the same driving noise $\widetilde{W}$ as used for the $\LTE$ scheme. Then
\begin{equation*}
\sup_{m=0,\ldots,M} \sup_{n=0,\ldots,N} \left(  \E \left[ | u^{\LTE}_{m,n} - u^{N}_{n}(t_{m})|^{2} \right] \right)^{1/2} \leq C(T,f,g) \sqrt{\Delta t^{2} + \frac{\Delta t}{\Delta x}}.
\end{equation*}
\end{proposition}
We change the driving noise for $u^{N}$ in Proposition~\ref{prop:mainStrong} to be able to derive strong convergence. The proof uses the standard arguments in order to apply a (discrete) Grönwall's lemma and is similar to the proof of strong convergence by the author and collaborators in \cite{MR4780408}. One important difference is that the solution $u$, the semi-discrete approximation $u^{N}$, and the $\LTE$ approximation $u^{\LTE}$ are all confined to $[a,b]$ in the current setting, implying that moment bounds are immediate.
\begin{proof}[Proof of Proposition~\ref{prop:mainStrong}]
Recall that $u^{\LTE}_{m,n} = u^{\LTE}(t_{m},x_{n})$ and that $u^{N}_{n}(t_{m}) = u^{N}(t_{m},x_{n})$. We use the integral equations~\eqref{eq:LTEmild} and~\eqref{eq:Nmild} with $t=t_{m}$. Note that, by definition, $\ell^{M}(t_{m})=t_{m}$ and so the integrals with the initial values in~\eqref{eq:LTEmild} and in~\eqref{eq:Nmild} coincide for $t=t_{m}$. We first split up the error into one term containing the deterministic integral and one term containing the stochastic integral:
\begin{equation*}
|u^{\LTE}(t_{m},x_{n}) - u^{N}(t_{m},x_{n})| \leq \left| I_{1} \right| + \left| I_{2} \right|,
\end{equation*}
where
\begin{equation*}
I_{1} = \int_{0}^{t_{m}} \int_{0}^{1} G^{N}(t_{m}-\ell^{M}(s),x_{n},y) f(u^{\LTE}(s,\kappa^{N}(y))) - G^{N}(t_{m}-s,x_{n},y) f(u^{N}(s,\kappa^{N}(y))) \diff s \diff y
\end{equation*}
and
\begin{equation*}
I_{2} = \int_{0}^{t_{m}} \int_{0}^{1} G^{N}(t_{m}-\ell^{M}(s),x_{n},y) g(u^{\LTE}(s,\kappa^{N}(y))) - G^{N}(t_{m}-s,x_{n},y) g(u^{N}(s,\kappa^{N}(y))) \diff \widetilde{W}(s,y).
\end{equation*}
As the estimates for $I_{2}$ containing the stochastic integral are the most interesting ones and as the estimates for $I_{1}$ are very similar, we leave the analogous computations for $I_{1}$ to the reader.

We first apply the triangle inequality to further split the error into one error term containing the time regularity of $G^{N}$ and one error term comparing $g(u^{N})$ to $g(u^{\LTE})$.
\begin{align*}
I_{2} &= \int_{0}^{t_{m}} \int_{0}^{1} \left( G^{N}( t_{m}-\ell^{M}(s),x_{n},y) - G^{N}(t_{m}-s,x_{n},y) \right) g(u^{\LTE}(s,\kappa^{N}(y))) \diff \widetilde{W}(s,y) \\ &+ \int_{0}^{t_{m}} \int_{0}^{1}  G^{N}(t_{m}-s,x_{n},y) \left( g(u^{\LTE}(s,\kappa^{N}(y))) - g(u^{N}(s,\kappa^{N}(y))) \right) \diff \widetilde{W}(s,y).
\end{align*}

Then, by Itô's isometry, we have
\begin{align*}
&\E \left[ \left|I_{2} \right|^{2} \right] \\ &\leq 2 \int_{0}^{t_{m}} \int_{0}^{1} \left| G^{N}(t_{m}-\ell^{M}(s),x_{n},y) - G^{N}(t_{m}-s,x_{n},y) \right|^{2} \E \left[ | g(u^{\LTE}(s,\kappa^{N}(y))) |^{2} \right] \diff s \diff y \\ &+ 2 \int_{0}^{t_{m}} \int_{0}^{1}  \left| G^{N}(t_{m}-s,x_{n},y) \right|^{2} \E \left[ | g(u^{\LTE}(s,\kappa^{N}(y))) - g(u^{N}(s,\kappa^{N}(y))) |^{2} \right] \diff s \diff y.
\end{align*}
Since $u^{\LTE}$ and $u^{N}$ only take values in $\CD = [a,b]$ (see Proposition~\ref{prop:LTE-BP} and Proposition~\ref{prop:N-BP}, respectively) and since $g$ satisfies Assumption~\ref{ass:g}, we can further estimate the above as
\begin{align*}
\E \left[ \left|I_{2} \right|^{2} \right] &\leq C(g) \int_{0}^{t_{m}} \int_{0}^{1} | G^{N}(t_{m}-\ell^{M}(s),x_{n},y) - G^{N}(t_{m}-s,x_{n},y)|^{2} \diff s \diff y \\ &+ C(g) \int_{0}^{t_{m}} \int_{0}^{1}  |G^{N}(t_{m}-s,x_{n},y)|^{2} \E \left[ |u^{\LTE}(s,\kappa^{N}(y)) - u^{N}(s,\kappa^{N}(y))|^{2} \right] \diff s \diff y.
\end{align*}
We use~\eqref{eq:intGNest} and~\eqref{eq:intGNdiffEst} to obtain the estimates
\begin{equation*}
\int_{0}^{t_{m}} \int_{0}^{1} \left| G^{N}(t_{m}-\ell^{M}(s),x_{n},y) - G^{N}(t_{m}-s,x_{n},y) \right|^{2} \diff s \diff y \leq \frac{\Delta t}{\Delta x}
\end{equation*}
and
\begin{equation*}
\int_{0}^{1} \left| G^{N}(t_{m}-s,x,y) \right|^{2} \diff y \leq C(T) \frac{1}{\sqrt{t_{m}-s}}.
\end{equation*}
Thus,
\begin{equation*}
\E \left[ \left|I_{2} \right|^{2} \right] \leq C(g) \frac{\Delta t}{\Delta x} + C(T,g) \int_{0}^{t_{m}} \frac{1}{\sqrt{t_{m}-s}} \sup_{k=0,\ldots,N} \E \left[ |u^{\LTE}(s,x_{k}) - u^{N}(s,x_{k})|^{2} \right] \diff s.
\end{equation*}

The above is not the correct form to apply a discrete Grönwall's inequality, to resolve this we split up as
\begin{equation}\label{eq:errsplit1}
  \begin{split}
    u^{\LTE}(s,x_{k}) - u^{N}(s,x_{k}) &= u^{\LTE}(\ell^{M}(s),x_{k}) - u^{N}(\ell^{M}(s),x_{k}) \\ &+ u^{\LTE}(s,x_{k}) - u^{\LTE}(\ell^{M}(s),x_{k}) \\ &+ u^{N}(\ell^{M}(s),x_{k}) - u^{N}(s,x_{k}).
  \end{split}
\end{equation}
We estimate the second term on the right hand side of~\eqref{eq:errsplit1} as
\begin{equation*}
\E \left[ \left| u^{\LTE}(s,x_{k}) - u^{\LTE}(\ell^{M}(s),x_{k}) \right|^{2} \right] \leq C(f,g) \left( \Delta t^{2} + \frac{\Delta t}{\Delta x} \right),
\end{equation*}
by Lemma~\ref{lem:uLTEtimeReg}, and we estimate the third term on the right hand side of~\eqref{eq:errsplit1} as
\begin{equation*}
\E \left[ \left| u^{N}(\ell^{M}(s),x_{k}) - u^{N}(s,x_{k}) \right|^{2} \right] \leq C(T,f,g) \frac{\Delta t}{\Delta x},
\end{equation*}
by the time regularity estimate in~\eqref{eq:NstReg}. Thus, we obtain the estimate
\begin{align*}
\E \left[ \left|I_{2} \right|^{2} \right] &\leq C(g) \frac{\Delta t}{\Delta x} + C(T,f,g) \left( \Delta t^{2} + \frac{\Delta t}{\Delta x} \right) \int_{0}^{t_{m}} \frac{1}{\sqrt{t_{m}-s}} \diff s \\ &+ C(T,g) \int_{0}^{t_{m}} \frac{1}{\sqrt{t_{m}-s}} \sup_{k=0,\ldots,N} \E \left[ |u^{\LTE}(\ell^{M}(s),x_{k}) - u^{N}(\ell^{M}(s),x_{k})|^{2} \right] \diff s \\ &\leq C(T,f,g) \left( \Delta t^{2} + \frac{\Delta t}{\Delta x} \right) \\ &+ C(T,g) \sum_{i=0}^{m-1} \sup_{k=0,\ldots,N} \E \left[ |u^{\LTE}(t_{i},x_{k}) - u^{N}(t_{i},x_{k})|^{2} \right] \int_{t_{i}}^{t_{i+1}} \frac{1}{\sqrt{t_{m}-s}} \diff s,
\end{align*}
where we also used the estimate
\begin{equation*}
\int_{0}^{t_{m}} \frac{1}{\sqrt{t_{m}-s}} \diff s \leq 2 T.
\end{equation*}
Lastly, we estimate
\begin{equation*}
\int_{t_{i}}^{t_{i+1}} \frac{1}{\sqrt{t_{m}-s}} \diff s = 2 \sqrt{t_{m}-t_{i}} - 2 \sqrt{t_{m}-t_{i+1}} = 2 \sqrt{t_{m}-t_{i}} \left( 1 - \sqrt{1-\frac{\Delta t}{t_{m}-t_{i}}} \right) \leq \frac{2 \Delta t}{\sqrt{t_{m}-t_{i}}},
\end{equation*}
where we also used that $1 - \sqrt{1-z} \leq z$ for all $z \in [0,1]$, to arrive at
\begin{equation*}
\E \left[ \left| I_{2} \right|^{2} \right] \leq C(T,f,g) \left( \Delta t^{2} + \frac{\Delta t}{\Delta x} \right)
+ C(T,g) \Delta t \sum_{i=0}^{m-1} \frac{1}{\sqrt{t_{m}-t_{i}}} \sup_{k=0,\ldots,N} \E \left[ |u^{\LTE}(t_{i},x_{k}) - u^{N}(t_{i},x_{k})|^{2} \right].
\end{equation*}
Similar arguments for $I_{1}$, yields in total
\begin{multline*}
\E \left[ |u^{\LTE}(t_{m},x_{n}) - u^{N}(t_{m},x_{n})|^{2} \right] \leq C(T,f,g) \left( \Delta t^{2} + \frac{\Delta t}{\Delta x} \right) \\ + C(T,f,g) \Delta t \sum_{i=0}^{m-1} \frac{1}{\sqrt{t_{m}-t_{i}}} \sup_{k=0,\ldots,N} \E \left[ |u^{\LTE}(t_{i},x_{k}) - u^{N}(t_{i},x_{k})|^{2} \right]
\end{multline*}
and an application of a Grönwall's inequality yields
\begin{equation*}
\sup_{n=0,\ldots,N} \E \left[ |u^{\LTE}(t_{m},x_{n}) - u^{N}(t_{m},x_{n})|^{2} \right] \leq C(T,f,g) \left( \Delta t^{2} + \frac{\Delta t}{\Delta x} \right).
\end{equation*}
Taking the supremum over $m=0,\ldots,M$ yields the desired result
\begin{equation*}
\sup_{m=0,\ldots,M} \sup_{n=0,\ldots,N} \left( \E \left[ |u^{\LTE}(t_{m},x_{n}) - u^{N}(t_{m},x_{n})|^{2} \right] \right)^{1/2} \leq C(T,f,g) \sqrt{\Delta t^{2}+\frac{\Delta t}{\Delta x}}.
\end{equation*}
\end{proof}
We are now in a position to prove the main theorem.
\begin{proof}[Proof of Theorem~\ref{th:mainWeak}]
We may, without loss of generality, assume that $\widetilde{W}$ is the driving noise for $u$. We use Jensen's inequality and Proposition~\ref{prop:mainStrong} to obtain the first statement
\begin{align*}
\left| \E \left[ F \left( u^{\LTE}(t_{m},x_{n}) \right) \right] - \E \left[ F \left( u(t_{m},x_{n}) \right) \right] \right| &\leq \left| \E \left[ \left| F \left( u^{\LTE}(t_{m},x_{n}) \right) - F \left( u(t_{m},x_{n}) \right) \right|^{2} \right] \right|^{1/2} \\&\leq C(F) \left| \E \left[ \left| u^{\LTE}(t_{m},x_{n}) - u(t_{m},x_{n}) \right|^{2} \right] \right|^{1/2} \\ &\leq C(F,T,f,g) \left( \Delta x^{1/2} + \sqrt{\Delta t^{2} + \frac{\Delta t}{\Delta x}} \right),
\end{align*}
where we also used the triangle inequality and the strong convergence estimate in~\eqref{eq:Nstrong}. For the second statement we note that if $\Delta t = \Delta x^{2}$ then $\frac{\Delta t}{\Delta x} = \sqrt{\Delta t}$.
\end{proof}

\section{Numerical experiments}\label{sec:num}
In this section, we verify the theoretical results in Section~\ref{sec:results} by numerical experiments. Recall that $N,\ M \in \mathbb{N}$ denote the spatial and temporal, respectively, discretisation parameters and that $\Delta x = 1/N$ and $\Delta t = T/M$. We use the coupling $\Delta t = \Delta x^{2}$ and we verify that the order of weak convergence is $1/4$ for global Lipschitz continuous test functions. Let us denote by $\Delta_{m} W^{N}$ the vector of size $N-1$ with elements $\Delta_{m,n} W^{N} = W^{N}_{n}(t_{m+1})-W^{N}_{n}(t_{m})$.

We consider three SPDEs with superlinearly growing multiplicative noise terms that fit within the considered framework described in Section~\ref{sec:setting}: An Allen--Cahn SPDE, a Nagumo SPDE and an SIS SPDE. SIS is an abbreviation for Susceptible-Inflected-Susceptible and is a compartment model epidemiology. We provide numerical results for the boundary preservation property and for the weak convergence errors of the $\LTE$ scheme for each of the aforementioned examples. In addition, we also provide numerical experiments illustrating that classical scheme do not preserve the boundary of the SPDEs. We compare boundary preservation of the $\LTE$ scheme defined in Section~\ref{sec:schemeconstruction} to the following classical schemes for the approximation of the system of SDEs in~\eqref{eq:uNSDE}:
\begin{itemize}
\item the Euler--Maruyama scheme (denoted EM below), see for instance \cite{MR1699161}
$$
u^{\EM}_{m+1}=u^{\EM}_m+ N^2 D^N u^{\EM}_{m} \Delta t+ f(u^{\EM}_{m}) \Delta t +\sqrt{N}g(u^{\EM}_m)\Delta_{m} W^{N},
$$
\item the semi-implicit Euler--Maruyama scheme (denoted SEM below), see for instance \cite{MR1699161}
$$
u^{\SEM}_{m+1}=u^{\SEM}_m+ N^2 D^N u^{\SEM}_{m+1} \Delta t + f(u^{\EM}_{m}) \Delta t + \sqrt{N} g(u^{\SEM}_m)\Delta_{m} W^{N},
$$
\item the exponential Euler scheme (denoted EXP below), see for instance \cite{MR3047942}
$$
u^{\SEXP}_{m+1}=e^{\Delta t N^2 D^N}\left(u^{\SEXP}_m+ f(u^{\SEXP}_{m}) \Delta t + \sqrt{N} g(u^{\SEXP}_m)\Delta_{m} W^{N} \right).
$$
\end{itemize}

We present in loglog plots the weak errors given by
\begin{equation}\label{eq:weakerr}
\sup_{m=0,\ldots,M} \sup_{n =0,\ldots,N} \left| \E \left[ F \left( u^{\LTE}_{m,n} \right) \right] - \E \left[ F \left( u^{ref}(t_{m},x_{n}) \right) \right] \right|,
\end{equation}
where the reference solution $u^{ref}$ is computed using the $\LTE$ scheme with $\Delta t^{ref} = 10^{-7}$, and we present boundary preservation and lack thereof in tables displaying the number of sample paths out of $100$ that only contained values in the domain $\CD$. We approximate the expected value in~\eqref{eq:weakerr} using $2500$ Monte Carlo samples. We have numerically verified that $2500$ Monte Carlo samples is sufficient for the Monte Carlo error to be small enough to observe the order of weak convergence when approximating the expected values in~\eqref{eq:weakerr}.

To show the order of weak convergence, we consider two choices of test functions, or their translations by a fixed constant, 
\begin{enumerate}
\item $F_{1}(r) = \exp(-|r|)$.
\item $F_{2}(r) = 0.5 \left( \sqrt{1-r^2} |r| + \arcsin (|r|)  \right)$.
\end{enumerate}
Note that $F_{1}$ and $F_{2}$ can we expressed as $\widetilde{F}(|r|)$ for some $\C^{1}$ function $\widetilde{F}$, and hence are globally Lipschitz continuous but not continuously differentiable everywhere. We use $F_{1}$ and $F_{2}$ for $[-1,1]$-valued solutions and we compose $F_{1}$ and $F_{2}$ with the translation map $r \in [0,1] \mapsto 2(r-\frac{1}{2}) \in [-1,1]$ for $[0,1]$-valued solutions to guarantee that the non-differentiability is in the domain and thus to obtain the true order of convergence.

\begin{example}[Allen--Cahn SPDE]
Consider the Allen--Cahn SPDE given by
\begin{equation}\label{eq:SAC}
\diff u(t) = \left( \Delta u(t) + u(t) - u(t)^{3} \right) \diff t + (1-u(t)^2) \diff W(t)
\end{equation}
subject to homogeneous Dirichlet boundary conditions; that is, $f(r) = r - r^3$ is cubic and $g(r) = 1 - r^2$ is quadratic. $f$ and $g$ satisfy the assumptions in Section~\ref{sec:spde}. We use $u_{0}(x)=  \sin (2 \pi x)$, for every $x \in [0,1]$, as initial function. The Allen--Cahn SPDE in~\eqref{eq:SAC} has as associated domain $\CD = [-1,1]$. This follows from Proposition~\ref{prop:u-BP} and that the inverse Lamperti transform of the associated SDE in~\eqref{eq:subprob1} is given by
\begin{equation*}
\Phi^{-1}(r) = \frac{(1+x_{0}) e^{2r}-(1-x_{0})}{(1+x_{0}) e^{2r}+(1-x_{0})},\ r \in \mathbb{R},
\end{equation*}
and only takes values in $(-1,1)$. See Appendix~\ref{app:SAC} for more details.

The deterministic Allen--Cahn equation \cite{ALLEN19791085} (with $f(r) = r - r^3$ and $g(r) = 0$) was first derived as a model for the time evolution of the interface between two phases of a material (a so-called evolving interface model). The states $\pm 1$ represents the pure states in the deterministic Allen--Cahn equation and its solutions can only take values in the domain $[-1,1]$. The stochastic Allen--Cahn equation with additive noise has been considered, for example, in the works \cite{MR3986273, MR3308418}. We consider noise of the form $g(u) \diff W = \left( 1-u^2 \right) \diff W$ to preserve the physical domain $\CD = [-1,1]$ of the deterministic Allen--Cahn equation. The Allen--Cahn type SDE (see~\eqref{eq:ACsde}) with drift coefficient $x - x^3$ and diffusion coefficient $1 - r^2$ has been considered in previous works by the author \cite{artBar,MR4737060}. To the best of our knowledge, as a result of the difficulty to treat the quadratic noise term, the stochastic Allen--Cahn equation~\eqref{eq:SAC} with noise of the form $\left( 1-u^2 \right) \diff W$ has not been considered before.

It is clear that the drift and diffusion coefficient functions $f$ and $g$ satisfy Assumptions~\ref{ass:f}, \ref{ass:g}, and~\ref{ass:fg}. We provide the implementation details of the $\LTE$ scheme applied to the Allen--Cahn equation~\eqref{eq:SAC} in Appendix~\ref{app:SAC}, as the purpose in this section is to provide the numerical results to support our theoretical results.

We first provide numerical experiments to illustrate that the $\LTE$ scheme is boundary-preserving and that the classical schemes $\EM$, $\SEM$ and $\SEXP$ are not. Table $1$ shows the proportion of $100$ sample paths for each of the considered scheme that only take values in the domain $\CD = [-1,1]$. 

\begin{table}[h!]
\begin{center}
\begin{tabular}{||c c c c c||} 
 \hline
 $\lambda$ & $\EM$ & $\SEM$ & $\SEXP$ & $\LTE$ \\ [0.5ex] 
 \hline\hline
 $1$ & $0/100$ & $100/100$ & $100/100$ & $100/100$ \\ 
 \hline
 $2$ & $0/100$ & $75/100$ & $94/100$ & $100/100$ \\ 
 \hline
 $3$ & $0/100$ & $18/100$ & $56/100$ & $100/100$ \\ [1ex]
 \hline
\end{tabular}
\caption{Proportion of samples containing only values in $\CD=[-1,1]$ out of $100$ simulated sample paths for the Euler--Maruyama scheme (EM), the semi-implicit Euler--Maruyama scheme (SEM), the exponential Euler scheme (EXP), and Lie--Trotter-Exact scheme ($\LTE$) for the stochastic Allen--Cahn equation in~\eqref{eq:SAC} for different choices of $\lambda>0$ and parameters $T=1$, $u_{0}(x) = \sin(2 \pi x)$, $\Delta x = 2^{-4}$ and $\Delta t = 2^{-2}$. \label{tb:SAC}}
\end{center}
\end{table}

Next, we present in Figure~\ref{num:SAC_conv_exp} the weak errors in~\eqref{eq:weakerr} for the $\LTE$ scheme. The obtained numerical orders in Figure~\ref{num:SAC_conv_exp} confirm the theoretical order derived in Theorem~\ref{th:mainWeak} with the coupling $\Delta t = \Delta x^{2}$.

\begin{figure}[!htbp]
\begin{center}
\advance\leftskip-3cm
\advance\rightskip-3cm
\includegraphics[width=0.8\textwidth]{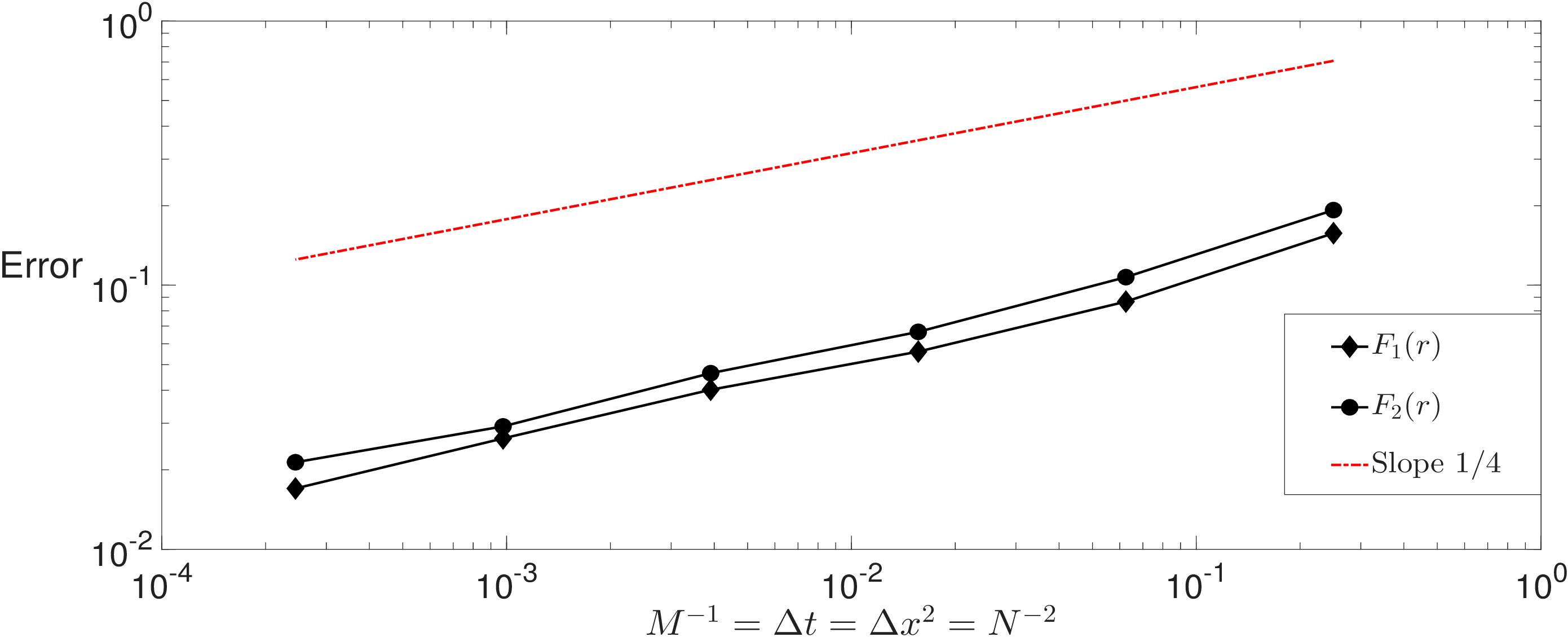}
\caption{Weak errors on the time interval $[0,T] = [0,1]$ of the $\LTE$ scheme for the Allen--Cahn SPDE for test functions $F_{1}(r) = e^{-|r|}$ and $F_{2}(r) = 0.5 \left( \sqrt{1-r^2} |r| + \arcsin(|r|) \right)$ and with $\lambda=1$, $u_{0}(x) = \sin (2 \pi x)$ and $\Delta t = \Delta x^{2}$. Averaged over $2500$ samples.}
\label{num:SAC_conv_exp}
\end{center}\end{figure}

\end{example}

\begin{example}[Nagumo SPDE]
Consider the Nagumo SPDE given by
\begin{equation}\label{eq:Nagumo}
\diff u(t) = \left( \Delta u(t) + u(t)(1-u(t))(u(t)-\gamma) \right) \diff t + u(t) (1-u(t)) \diff W(t),
\end{equation}
for some $\gamma \in (0,1/2)$, subject to non-homogeneous Dirichlet boundary conditions equal to $1/2$; that is, $f(r) = r (1-r) (r-\gamma)$ is cubic and $g(r) = r (1-r)$ is quadratic. It is clear that $f$ and $g$ satisfy the assumptions in Section~\ref{sec:spde}. We use $u_{0}(x) = 0.5 \left( \sin (\pi x) + 1/2 \right)$. The Nagumo SPDE in~\eqref{eq:Nagumo} has as associated domain $\CD = [0,1]$. Similarly to the stochastic Allen--Cahn equation, this follows from Proposition~\ref{prop:u-BP}, the change of variables presented in Appendix~\ref{app:Nagumo} and that the inverse Lamperti transform of the Nagumo type SDE in~\eqref{eq:NagumoSDETC} given by
\begin{equation*}
\Phi^{-1}(r) = \frac{(1+x_{0}) e^{2r}-(1-x_{0})}{(1+x_{0}) e^{2r}+(1-x_{0})},\ r \in \mathbb{R},
\end{equation*}
only takes values in $(-1,1)$.

The motivation for the (deterministic) Nagumo equation \cite{MCKEAN1970209} comes from the modelling of the voltage in the axon in a neuron and can only take values in $\CD = [0,1]$ \cite{MCKEAN1970209}. The stochastic Nagumo equation with additive noise has been considered in, for example, \cite{MR3308418}. We consider the noise of the from $g(u) \diff W = u (1-u) \diff W$ to preserve the invariant domain $[0,1]$ of the deterministic Nagumo equation. To the best of our knowledge, as a result of the difficulty of treat the quadratically growing noise term, the stochastic Nagumo equation with noise of the form $u (1-u) \diff W$ has not been considered before.

We present in Table~\ref{tb:Nagumo} the proportion of $100$ samples of the considered schemes that only take values in the domain $[0,1]$. Table~\ref{tb:Nagumo} verifies that the $\LTE$ is boundary-preserving and that the classical schemes $\EM$, $\SEM$ and $\SEXP$ are not.

\begin{table}[h!]
\begin{center}
\begin{tabular}{||c c c c c||} 
 \hline
 $\lambda$ & $\EM$ & $\SEM$ & $\SEXP$ & $\LTE$ \\ [0.5ex] 
 \hline\hline
 $2$ & $0/100$ & $100/100$ & $100/100$ & $100/100$ \\ 
 \hline
 $3$ & $0/100$ & $91/100$ & $98/100$ & $100/100$ \\ 
 \hline
 $4$ & $0/100$ & $51/100$ & $76/100$ & $100/100$ \\ [1ex]
 \hline
\end{tabular}
\caption{Proportion of samples containing only values in $\CD=[0,1]$ out of $100$ simulated sample paths for the Euler--Maruyama scheme ($\EM$), the semi-implicit Euler--Maruyama scheme ($\SEM$), the tamed Euler scheme (TE), and Lie--Trotter-Exact scheme ($\LTE$) for the stochastic Nagumo equation in~\eqref{eq:Nagumo} for different choices of $\lambda>0$ and parameters $T=1$, $x_{0}= 0.5 (\sin (\pi x + 1/2))$, $\Delta x = 2^{-4}$ and $\Delta t = 2^{-2}$. \label{tb:Nagumo}}
\end{center}
\end{table}

To numerically verify Theorem~\ref{th:mainWeak}, the main convergence result, we present the weak errors in~\eqref{eq:weakerr} of the $\LTE$ scheme in Figure~\ref{num:Nagumo_conv_exp}. The orders in Figure~\ref{num:Nagumo_conv_exp} confirm the derived theoretical order in Theorem~\ref{th:mainWeak}.

\begin{figure}[!htbp]
\begin{center}
\advance\leftskip-3cm
\advance\rightskip-3cm
\includegraphics[width=0.8\textwidth]{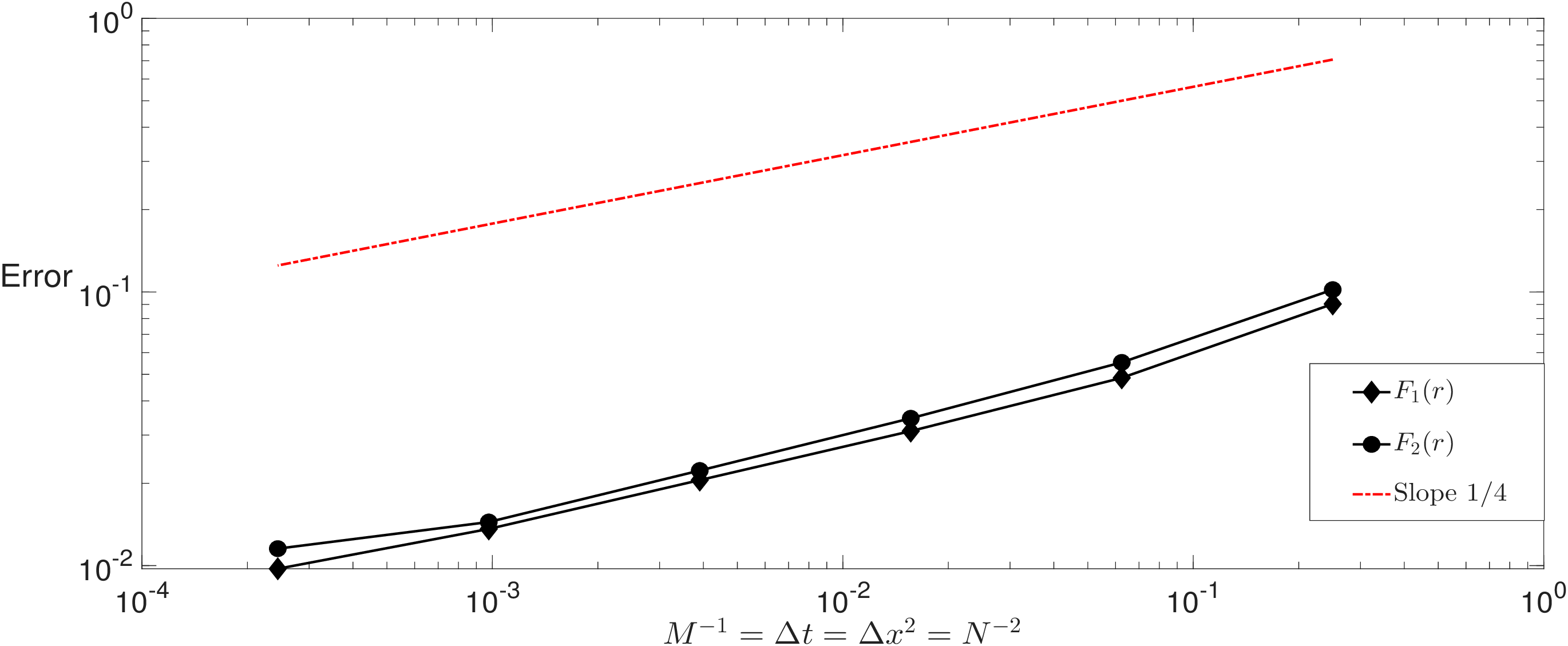}
\caption{Weak errors on the time interval $[0,T] = [0,1]$ of the $\LTE$ scheme for the Nagumo SPDE using the test functions $F_{1}(r) = e^{-|r|}$ and $F_{2}(r) = 0.5 \left( \sqrt{1-r^2} |r| + \arcsin(|r|) \right)$ and with $\lambda=1$, $u_{0}(x) = 0.5( \sin (\pi x) + 1/2)$ and $\Delta t = \Delta x^{2}$. Averaged over $2500$ samples.}
\label{num:Nagumo_conv_exp}
\end{center}\end{figure}
\end{example}

\begin{example}[SIS SPDE]
Consider the SIS SPDE given by
\begin{equation}\label{eq:SIS}
\diff u(t) = \left( \Delta u(t) + u(t) \left(1- u(t) \right) \right) \diff t + u(t)(1-u(t)) \diff W(t)
\end{equation}
subject to non-homogeneous Dirichlet boundary conditions equal to $1/2$; that is, $f(r) = r (1-r)$ and $g(r) = r (1-r)$ are quadratic. $f$ and $g$ satisfy the assumptions in Section~\ref{sec:setting}. We use $u_{0}(x) = 0.5 \left( \sin (\pi x) + 1/2 \right)$. The SIS SPDE~\eqref{eq:SIS} has as associated domain $\CD = [0,1]$. This follows from Proposition~\ref{prop:u-BP}, the change of variables used in Appendix~\ref{app:SIS} and that the inverse Lamperti transform of the SIS type SDE in~\eqref{eq:SIS} is given by
\begin{equation*}
\Phi^{-1}(r) = \frac{(1+x_{0}) e^{2r}-(1-x_{0})}{(1+x_{0}) e^{2r}+(1-x_{0})},\ r \in \mathbb{R},
\end{equation*}
and only takes values in $(-1,1)$.

The motivation for the SIS stochastic differential equation comes from the modelling of the spread of some disease in a population \cite{KERMACK199133} and can only take values in $\CD = [0,1]$ (as it represents the proportion of the population with the disease). For example, the SDE with drift and diffusion coefficient functions both equal to $r (1-r)$ is an instance of an SIS SDE. Such SDEs has been considered in the works \cite{GraySIS, DomPres, MR4737060}. Similarly, we consider noise of the from $u (1-u) \diff W$ to preserve the physical domain $\CD = [0,1]$. To the best of our knowledge, the SIS SPDE~\eqref{eq:SIS} with noise of the form $(1-u^{2}) \diff W$ has not been considered before.

Firstly, we present numerical results in Table~\ref{tb:SIS} to confirm that the $\LTE$ scheme is boundary-preserving and that the considered classical schemes are not. We present in Table~\ref{tb:SIS} the proportion out of $100$ samples, for each scheme, that only takes values in the domain $\CD = [0,1]$.

\begin{table}[h!]
\begin{center}
\begin{tabular}{||c c c c c||} 
 \hline
 $\lambda$ & $\EM$ & $\SEM$ & $\SEXP$ & $\LTE$ \\ [0.5ex] 
 \hline\hline
 $2$ & $0/100$ & $100/100$ & $100/100$ & $100/100$ \\ 
 \hline
 $3$ & $0/100$ & $89/100$ & $98/100$ & $100/100$ \\ 
 \hline
 $4$ & $0/100$ & $78/100$ & $43/100$ & $100/100$ \\ [1ex]
 \hline
\end{tabular}
\caption{Proportion of samples containing only values in $\CD=[0,1]$ out of $100$ simulated sample paths for the Euler--Maruyama scheme ($\EM$), the semi-implicit Euler--Maruyama scheme ($\SEM$), the exponential Euler scheme (EXP), and Lie--Trotter-Exact scheme ($\LTE$) for the SIS SPDE in~\eqref{eq:SIS} for different choices of $\lambda>0$ and parameters $T=1$, $u_{0}(x)= 0.5 (\sin(\pi x) + 1/2)$, $\Delta x = 2^{-4}$ and $\Delta t = 2^{-2}$. \label{tb:SIS}}
\end{center}
\end{table}

Secondly, we present the weak errors in~\eqref{eq:weakerr} of the $\LTE$ scheme in Figure~\ref{num:SIS_conv_exp}. We use the two considered test functions $F_{1}$ and $F_{2}$ to numerically verify the theoretical convergence order obtained in Theorem~\ref{th:mainWeak}.

\begin{figure}[!htbp]
\begin{center}
\advance\leftskip-3cm
\advance\rightskip-3cm
\includegraphics[width=0.8\textwidth]{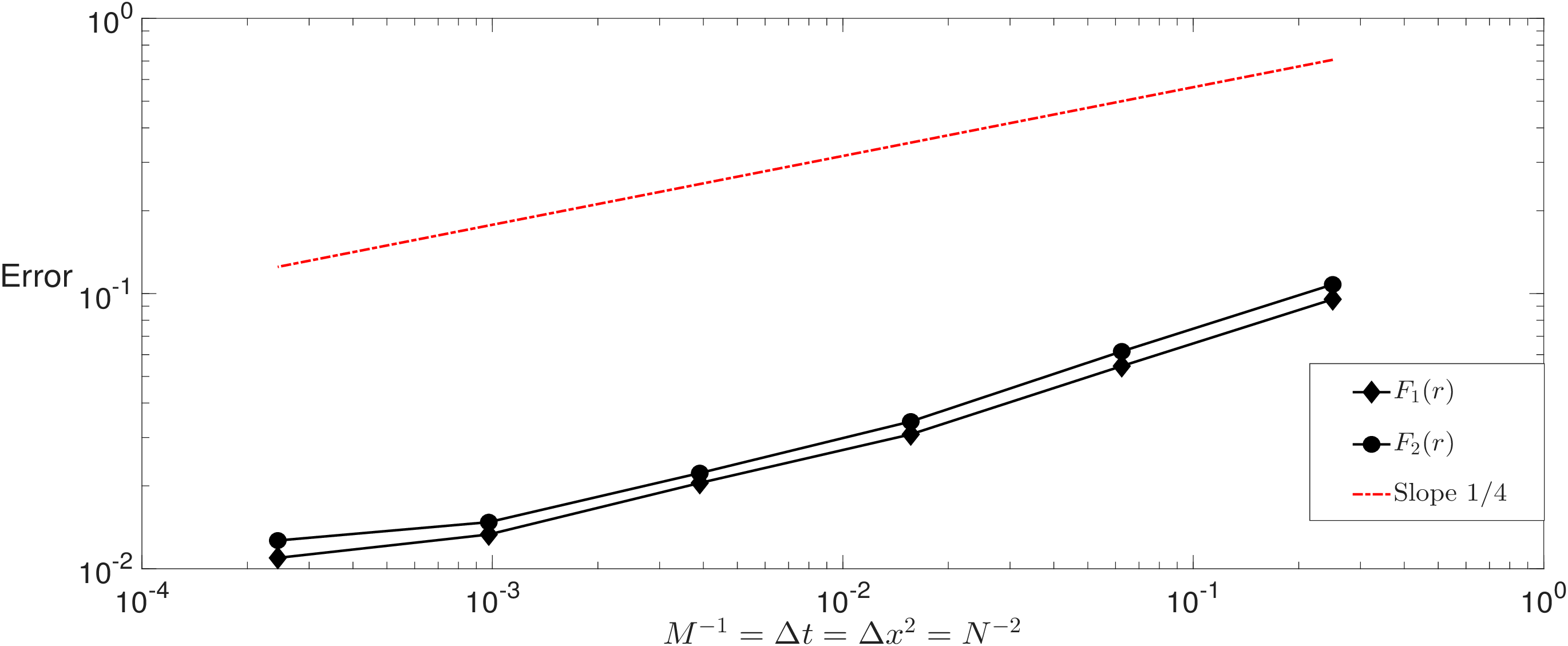}
\caption{Weak errors on the time interval $[0,T] = [0,1]$ of the $\LTE$ scheme for the SIS SPDE for the test functions $F_{1}(r) = e^{-|r|}$ and $F_{2}(r) = 0.5 \left( \sqrt{1-r^2} |x| + \arcsin(|r|) \right)$ and with $\lambda=1$, $u_{0}(x) = 0.5( \sin (\pi x) + 1/2)$ and $\Delta t = \Delta x^{2}$. Averaged over $2500$ samples.}
\label{num:SIS_conv_exp}
\end{center}\end{figure}
\end{example}

\begin{appendix}
\section{Additional proofs}\label{sec:app-proofs}
\subsection{About the spatial discretisation $u^{N}$}\label{sec:app-uN}
Here we provide the proof of Proposition~\ref{prop:N-BP}.
\begin{proof}[Proof of Proposition~\ref{prop:N-BP}]
It suffices to show 
\begin{equation*}
\PP \left( u^{N}_{n}(t) \in \CD,\ \forall t \in [0,T],\ \forall n=0,\ldots,N \right) = 1,
\end{equation*}
since $x \mapsto u^{N}(t,x)$ is piecewise linear and satisfies $u^{N}(t,x_{n}) = u^{N}_{n}(t)$ for $n = 0,\ldots,N$.

The argument is based on extending and truncating the coefficient functions $f$ and $g$ as follows. Since $f(a)=f(b)=g(a)=g(b)=0$, we can extend $f$ and $g$ such that, for some $\epsilon>0$, $f \in \C^{2}(\CD_{\epsilon})$ and $g \in \C^{3}(\CD_{\epsilon})$ where $\CD_{\epsilon} = [l-\epsilon,r+\epsilon]$. Let $L \in \mathbb{N}$ sufficiently large and define $f_{L}$ and $g_{L}$ as follows
\begin{enumerate}[(i)]
\item $f_{L}(r) = f(r)$ and $g_{L}(r) = g(r)$ for every $r \in \CD_{1/L}$.
\item $f_{L}(r)=0$ and $g_{L}(r)=0$ on $\mathbb{R} \setminus \CD_{2/L}$.
\item $f_{L}$ and $g_{L}$ are linear on $\CD_{2/L} \setminus \CD_{1/L}$.
\end{enumerate}
Note that we take $L \in \mathbb{N}$ large enough so that $\CD_{1/L} \subset \CD_{\epsilon}$. Then, $f_{L}$ and $g_{L}$ are, for every $L \in \mathbb{N}$, globally Lipschitz continuous functions that coincide with $f$ and $g$, respectively, on $\CD_{1/L}$ and are identically zero outside $\CD_{2/L}$. We now consider the solution $u_{L}^{N}$ of the following system of SDEs
\begin{equation}\label{eq:app-truncSDE}
\diff u^{N}_{L}(t) = \left( N^{2} D^{N} u^{N}_{L}(t) + f_{L}(u^{N}_{L}(t)) \right) \diff t + \sqrt{N} g_{L}(u^{N}_{L}(t)) \diff W^{N}(t)
\end{equation}
with initial vector $u^{N}_{L}(0) = u^{N}(0)$ and with globally Lipschitz continuous coefficient functions. We next show that $\PP \left( u^{N}_{L} \in \CD \right) = 1$ for every $L \in \mathbb{N}$. Let us denote by $\hat{u}^{\LTE}_{L}$ the $\LTE$ scheme defined in Section~\ref{sec:schemeconstruction} applied to the truncated system of SDEs in~\eqref{eq:app-truncSDE} with exact samples instead of exact simulations to solve the diagonal system of SDEs in~\eqref{eq:subprob1}. Then $\PP \left( \hat{u}^{\LTE}_{L} \in \CD \right) = 1$, by construction, and a modification of the proof of Proposition~\ref{prop:mainStrong} implies that $\hat{u}^{\LTE}_{L} \to u^{N}_{L}$ strongly as $\Delta t \to 0$. Note that, contrary to argument presented in the proof of Proposition~\ref{prop:mainStrong}, there is no need to assume that $\PP(u^{N}_{L} \in \CD)=1$, as the coefficient functions are now globally Lipschitz continuous for the truncated system of SDEs in~\eqref{eq:app-truncSDE}. Since strong convergence implies almost sure convergence along a subsequence, we conclude that $\PP(u^{N}_{L} \in \CD)=1$. Then, as modifying $f_{L}$ and $g_{L}$ outside $\CD$ does not affect the solution $u_{L}^{N}$, it follows that $u_{L}^{N}$ solves the original system of SDE in~\eqref{eq:FDeq} and hence $u^{N} = u^{N}_{L}$. We conclude that $\PP(u^{N} \in \CD)=1$. Note that any specific choice of $L$ would be sufficient to conclude that $\PP(u^{N} \in \CD)=1$.
\end{proof}

\subsection{About the exact solution $u$}\label{sec:app-EUB}
Here we provide the proof of Proposition~\ref{prop:u-BP}.
\begin{proof}[Proof of Proposition~\ref{prop:u-BP}]
Standard theory for SPDEs assume globally Lipschitz continuous coefficient functions $f$ and $g$. We make no claim that this is the first proof of well-posedness of continuous mild solutions to the SPDE in~\eqref{intro:SPDE}, we provide the following argument as it fits well into the framework developed in this paper and that, to the best of our knowledge, the unique continuous mild solution only takes values in $\CD$ is not known in the literature. We prove the following statements:
\begin{enumerate}
\item There exists a continuous mild solution $u$ to the SPDE in~\eqref{intro:SPDE}.
\item Any continuous mild solution $u$ to the SPDE in~\eqref{intro:SPDE} must satisfy $\PP \left( u \in \CD \right)=1$.
\item There exists a unique continuous mild solution $u$ to the SPDE in~\eqref{intro:SPDE}.
\end{enumerate}
As in the proof in Appendix~\ref{sec:app-uN}, we extend and truncate the coefficient functions $f$ and $g$ as follows. We recall the extension and truncation here for completeness. We extend $f$ and $g$ such that, for some $\epsilon>0$, $f \in \C^{2}(\CD_{\epsilon})$ and $g \in \C^{3}(\CD_{\epsilon})$ where $\CD_{\epsilon} = [l-\epsilon,r+\epsilon]$. Let $L \in \mathbb{N}$ be sufficiently large and define $f_{L}$ and $g_{L}$ as follows
\begin{enumerate}[(i)]
\item $f_{L}(r) = f(r)$ and $g_{L}(r) = g(r)$ for every $r \in \CD_{1/L}$.
\item $f_{L}(r)=0$ and $g_{L}(r)=0$ on $\mathbb{R} \setminus \CD_{2/L}$.
\item $f_{L}$ and $g_{L}$ are linear on $\CD_{2/L} \setminus \CD_{1/L}$.
\end{enumerate}
Then, $f_{L}$ and $g_{L}$ are, for every $L \in \mathbb{N}$, globally Lipschitz continuous functions that coincide with $f$ and $g$, respectively, on $\CD_{1/L}$ and are identically zero outside $\CD_{2/L}$. We consider the following truncated SPDE
\begin{equation}\label{eq:app-truncSPDE}
\diff u_{L}(t) = \left(\Delta u(t) + f_{L}(u_{L}(t)) \right) \diff t + g_{L}(u_{L}(t)) \diff W(t)
\end{equation}
with initial state $u_{L}(0) = u_{0}$ and with globally Lipschitz continuous coefficient functions $f_{L}$ and $g_{L}$. Existence and uniqueness of a continuous mild solution $u_{L}$ to the truncated SPDE~\eqref{eq:app-truncSPDE} follows from standard SPDE theory, since $f_{L}$ and $g_{L}$ are globally Lipschitz continuous.

We start with $(1)$. We first show that $\PP(u_{L} \in \CD)=1$, for every $L \in \mathbb{N}$. It then follows that $u_{L}$ solves the original SPDE in~\eqref{intro:SPDE}, as modifying $f_{L}$ and $g_{L}$ outside $\CD$ does not affect the solution $u_{L}$, and establishes existence of a continuous mild solution to the original SPDE in~\eqref{intro:SPDE}. To this end, let us denote by $u^{N}_{L}$ the solution of the following system of SDEs
\begin{equation*}
\diff u^{N}_{L}(t) = \left( N^{2} D^{N} u^{N}_{L}(t) + f_{L}(u^{N}_{L}(t)) \right) \diff t + \sqrt{N} g_{L}(u^{N}_{L}(t)) \diff W^{N}(t).
\end{equation*}
Then, as was shown in the proof of Propisition~\ref{prop:N-BP}, $\PP \left( u^{N}_{L} \in \CD \right) = 1$ and Theorem $3.1$ in~\cite{MR1644183} implies that $u^{N}_{L} \to u_{L}$ strongly as $\Delta t \to 0$. Since strong convergence implies almost sure convergence along a subsequence, we conclude that $\PP(u_{L} \in \CD)=1$. This, in turn, implies $(1)$.

We next prove $(2)$ that any continuous mild solution $u$ to~\eqref{sec:intro} must satisfy $\PP(u \in \CD)=1$. To this end, suppose that $u$ is a continuous mild solution to~\eqref{sec:intro}.  We define
\begin{equation*}
\tau_{k} = \inf \{ t>0:\ \exists x \in [0,1] \text{ such that } u(t,x) \not\in [l-1/k,l+1/k] \},
\end{equation*}
for $k \in \mathbb{N}$, and we let
\begin{equation*}
\tau = \inf \{ t>0:\ \exists x \in [0,1] \text{ such that } u(t,x) \not\in [a,b] \} = \inf_{k \in \mathbb{N}} \tau_{k}.
\end{equation*}
We now claim that $\PP(\tau_{k}=\infty) = 1$ for every $k \in \mathbb{N}$. Observe that this would imply that $\PP( \tau=\infty) = 1$, and hence that $\PP(u \in \CD) = 1$. We prove the claim by contradiction; that is, suppose that $\PP(\tau_{k} < \infty) = \epsilon > 0$ for some $k \in \mathbb{N}$. Let $u_{k}$ be the solution to the truncated SPDE in~\eqref{eq:app-truncSPDE} with $L = k$ and let $\omega \in \Omega$ be a sample point such $\tau_{k}(\omega) < \infty$. We add $\omega$ to indicate the chosen sample point $\omega \in \Omega$. Recall that $u_{k} \in \CD$ with probability $1$ (see proof of Proposition~\ref{prop:N-BP}) and that $(t,x) \mapsto u(t,x)$ is (jointly) continuous with probability $1$ by assumption. Then, by definition of $\tau_{k}$, we have that
\begin{equation*}
u(\omega;t,x) \in [l - 1/k,r + 1/k],
\end{equation*}
for every $t < \tau_{k}(\omega)$ and for every $x \in [0,1]$. Then, by uniqueness for SPDEs with globally Lipschitz continuous coefficients, we have that
\begin{equation*}
u(\omega;t,x) = u_{k}(\omega;t,x),
\end{equation*}
for $t < \tau_{k}(\omega)$ and for $x \in [0,1]$. But then, since $\PP(u_{k} \in \CD)=1$, we obtain that
\begin{equation*}
u(\omega;t,x) \in [a,b]
\end{equation*}
for $t < \tau_{k}(\omega)$ and for $x \in [0,1]$. By continuity of the sample paths of $u$, we deduce that
\begin{equation*}
u(\omega;t,x) \in [a,b]
\end{equation*}
for $t \leq \tau_{k}(\omega)$ and for $x \in [0,1]$. This is a contradiction to the definition of $\tau_{k}$. We conclude that $\PP(\tau_{k} < \infty) = 0$ for every $k \in \mathbb{N}$, or equivalently that $\PP(\tau_{k} = \infty) = 1$ for every $k \in \mathbb{N}$. Thus, we conclude the desired property in $(2)$ that $\PP(u \in \CD)=1$.

Lastly, let $u$ and $v$ be continuous mild solutions of the SPDE in~\eqref{intro:SPDE}. Then, by $(2)$, $\PP(u \in \CD)=1$ and $\PP(v \in \CD)=1$. Therefore, $u$ and $v$ both solves the SPDE with globally Lipschitz continuous coefficients in~\eqref{eq:app-truncSPDE} and we can conclude that $u=v$ by uniqueness of solutions of~\eqref{eq:app-truncSPDE}. This gives $(3)$ and concludes the proof.
\end{proof}

\subsection{Time-change for SDEs}\label{sec:app-TC}
Here we provide some extra details for the deterministic time-change of SDEs considered in Section~\ref{sec:exactSim}. This is done as follows: We apply a time-change based on $\Theta (t) = \lambda^{-2} t$ by inserting $\Theta$ into the SDE in equation~\eqref{eq:SDE-exactSim} that we recall here
\begin{equation}\label{eq:SDE-TC-app}
\left\lbrace
\begin{aligned}
& \diff X(t) = f(X(t)) \diff t + \lambda g(X(t)) \diff B(t),\ t \in (0,T] \\ 
& X(0) = x_{0} \in \D,
\end{aligned}
\right.
\end{equation}
for some $\lambda > 0$, where $f$ and $g$ satisfy the assumptions in Section~\ref{sec:spde}. We insert $\Theta(t) = \lambda^{-2} t$ into equation~\eqref{eq:SDE-TC-app} to see that
\begin{equation*}
X(\Theta(t)) = x_{0} + \int_{0}^{\Theta(t)} f(X(s)) \diff s + \int_{0}^{\Theta(t)} \lambda g(X(s)) \diff B(s).
\end{equation*}
We apply a standard change of variables to re-write the deterministic integral (with stochastic integrand) as
\begin{equation*}
\int_{0}^{\Theta(t)} f(X(s)) \diff s = \frac{1}{\lambda^{2}} \int_{0}^{t} f(X(\Theta(s))) \diff s.
\end{equation*}
For the stochastic integral, we use the property that
\begin{equation*}
\int_{0}^{\Theta(t)} \lambda g(X(s)) \diff B(s) = \int_{0}^{t} g(X(\Theta(s))) \diff B(s)
\end{equation*}
holds in the weak sense (in law). See \cite{MR2001996} for more details on this. Thus, in total we have that
\begin{equation*}
\widetilde{X}(t) = X(\Theta(t)) = x_{0} + \frac{1}{\lambda^{2}} \int_{0}^{t} f(X(\Theta(s))) \diff s + \int_{0}^{t} g(X(\Theta(s))) \diff B(s)
\end{equation*}
in the weak sense, which is equivalent to $\widetilde{X}$ satisfying the desired SDE in equation~\eqref{eq:SDE-exactSim-TC} that we recall here
\begin{equation*}
\left\lbrace
\begin{aligned}
& \diff \widetilde{X}(t) = \lambda^{-2} f(\widetilde{X}(t)) \diff t + g(\widetilde{X}(t)) \diff B(t),\ t \in (0,\lambda^{2} T], \\ 
& \widetilde{X}(0) = x_{0} \in \D.
\end{aligned}
\right.
\end{equation*}

\section{Implementation details of the $\LTE$ scheme}\label{app:implementation}
In this section, we provide the implementation details of the $\LTE$ scheme constructed in Section~\ref{sec:schemeconstruction} when applied to the stochastic Allen--Cahn equation in~\eqref{eq:SAC}, to the stochastic Nagumo equation in~\eqref{eq:Nagumo} and to the SIS SPDE in~\eqref{eq:SIS}. The $\LTE$ can be directly applied to the stochastic Allen--Cahn equation, as we have homogeneous Dirichlet boundary conditions in this case. On the other hand, the considered stochastic Nagumo equation and SIS SPDE have non-homogeneous Dirichlet boundary conditions and we first apply a change of variables to obtain equivalent SPDEs with homogeneous Dirichlet boundary conditions.

\subsection{Stochastic Allen--Cahn equation}\label{app:SAC}
The implementation details of the $\LTE$ scheme applied to the stochastic Allen--Cahn equation~\eqref{eq:SAC} that is not included in the main part of the work is the remaining details for the exact simulation of the Allen--Cahn type SDEs
\begin{equation}\label{eq:ACsde}
\left\lbrace
\begin{aligned}
& \diff X(t) = \left( X(t) - X(t)^3 \right) + \lambda \left( 1 - X(t)^2 \right) \diff B(t),\ t \in (0,T], \\ 
& X(0) = x_{0} \in (-1,1),
\end{aligned}
\right.
\end{equation}
with solution process $X$. Note that, for the same reason as in Section~\ref{sec:exactSim}, we scale the noise with a parameter $\lambda>0$. For the implementation of the $\LTE$ scheme we use $\lambda = \sqrt{N}$ where $N$ is the finite difference discretisation parameter. The first step is to apply a deterministic time-change (see Section~\ref{sec:app-TC}) to transfer the constant $\lambda$ into the drift. The process $\widetilde{X}(t) = X(\lambda^{-2} t)$ satisfies in the weak sense the Itô SDE
\begin{equation}\label{eq:ACsdeTC}
\left\lbrace
\begin{aligned}
& \diff \widetilde{X}(t) = \lambda^{-2} \left( \widetilde{X}(t) - \widetilde{X}(t)^3 \right) \diff t + \left( 1 - \widetilde{X}(t)^2 \right) \diff B(t),\ t \in (0,T_{\lambda}], \\ 
& \widetilde{X}(0) = x_{0} \in (-1,1),
\end{aligned}
\right.
\end{equation}
where we recall that $T_{\lambda} = \lambda^{2} T$. Sampling $X(\Delta t)$ starting from $X(0) = x_{0}$ is equivalent to sampling $\widetilde{X}(\lambda^{2} \Delta t)$ starting from $\widetilde{X}(0) = x_{0}$. We compute the Lamperti transform (see equation~\eqref{eq:LampTrans}) of the time-changed SDE in~\eqref{eq:ACsdeTC} 
\begin{equation*}
\Phi(r) = \int_{x_{0}}^{r} \frac{1}{1-w^2} \diff w = \frac{1}{2} \left( \log(1+r) - \log(1-r) - \log(1+x_{0}) + \log(1-x_{0}) \right),\ r \in \D,
\end{equation*}
and the inverse of the Lamperti transform
\begin{equation*}
\Phi^{-1}(r) = \frac{(1+x_{0})e^{2r} - (1-x_{0})}{(1+x_{0})e^{2r} + (1-x_{0})},\ r \in \mathbb{R}.
\end{equation*}
The transformed process $Y(t) = \Phi (\widetilde{X}(t))$ then satisfies
\begin{equation*}
\left\lbrace
\begin{aligned}
& \diff Y(t) = \alpha(Y(t)) \diff t + \diff B(t),\ t \in (0,T_{\lambda}], \\ 
& Y(0) = \Phi(x_{0}) \in \mathbb{R},
\end{aligned}
\right.
\end{equation*}
with drift coefficient function given by
\begin{equation*}
\alpha(r) = \left( 1 + \frac{1}{\lambda^{2}} \right) \Phi^{-1}(r),\ r \in \mathbb{R}.
\end{equation*}
We compute $\alpha'$ by direct differentiation and using~\eqref{eq:diffPhiinv}
\begin{align*}
\alpha'(r) &= \left( 1 + \frac{1}{\lambda^{2}} \right) \frac{\diff}{\diff r} \Phi^{-1}(r) = \left( 1 + \frac{1}{\lambda^{2}} \right) g \left( \Phi^{-1}(r) \right) \\ &= \left( 1 + \frac{1}{\lambda^{2}} \right) \frac{4 (1-x_{0})(1+x_{0}) e^{2r}}{\left( (1+x_{0})e^{2r} + (1-x_{0}) \right)^{2}},\ r \in \mathbb{R},
\end{align*}
and we recall that $\phi$ is defined to be
\begin{equation*}
\phi(r) = \frac{1}{2} \alpha'(r) + \frac{1}{2} \alpha(r)^{2} - k_{1},\ r \in \mathbb{R}.
\end{equation*}
In this case, $k_{1} = 0$ and
\begin{equation*}
k_{2} = \frac{1}{2} \left( 1 + \frac{1}{\lambda^{2}} \right) + \frac{1}{2} \left( 1 + \frac{1}{\lambda^{2}} \right)^{2},
\end{equation*}
since the above allows us to bound
\begin{equation}\label{eq:ACsdealphaalphapest}
0 \leq \alpha'(r) \leq 1 + \frac{1}{\lambda^{2}},\ 0 \leq \alpha^{2}(r) \leq \left( 1 + \frac{1}{\lambda^{2}} \right)^{2},
\end{equation}
both for every $r \in \mathbb{R}$. The first estimate in~\eqref{eq:ACsdealphaalphapest} follows from
\begin{equation*}
0 \leq \inf_{w \in [-1,1]} g(w) \leq g \left( \Phi^{-1}(r) \right) \leq \sup_{w \in [-1,1]} g(w),
\end{equation*}
for every $r \in \mathbb{R}$, and the second estimate in~\eqref{eq:ACsdealphaalphapest} follows from the property that $\Phi^{-1}(r) \in [-1,1]$ for every $r \in \mathbb{R}$. We next compute $A$ in~\eqref{eq:Adef}
\begin{equation*}
A(r) = \left( 1 + \frac{1}{\lambda^{2}} \right) \int_{0}^{r} \Phi^{-1}(w) \diff w = \left( 1 + \frac{1}{\lambda^{2}} \right) \left( \log \left( (1+x_{0})e^{2r} + (1-x_{0}) \right) - r - \log(2) \right)
\end{equation*}
and use that the target density $h$ in~\eqref{eq:hdef} is proportional to 
\begin{equation}
h(r) \propto \exp \left( A(r) - \frac{r^{2}}{2 T_{\lambda}} \right) \propto \left( (1+x_{0}) e^{r} + (1-x_{0})e^{-r} \right)^{1 + \frac{1}{\lambda^{2}}} e^{-\frac{r^{2}}{2 T_{\lambda}}}.
\end{equation}
As we cannot directly sample from $h$, we use an accept-reject strategy to sample for $h$. Keeping in mind that we let $\frac{1}{\lambda^{2}} \to 0$, we use the proposal density $H_{\nu}$ given by
\begin{equation*}
H_{\nu}(r) \propto \left( (1+x_{0}) e^{r} + (1-x_{0}) e^{-r} \right) e^{- \frac{\nu r^{2}}{2 T_{\lambda}}},
\end{equation*}
for some $\nu > 0$ to control the tails of $H_{\nu}$ in relation to $h$. Owing to the factors $e^{-\frac{r^{2}}{2 T_{\lambda}}}$ in $h$ and  $e^{- \frac{\nu r^{2}}{2 T_{\lambda}}}$ in $H_{\nu}$, $h$ and $H_{\nu}$ are probability densities and there exists normalising constants $c_{h}$ and $c_{H_{\nu}}$ such that
\begin{equation*}
h(r) = \frac{1}{c_{h}} \left( (1+x_{0}) e^{r} + (1-x_{0})e^{-r} \right)^{1 + \frac{1}{\lambda^{2}}} e^{-\frac{r^{2}}{2 T_{\lambda}}},\ r \in \mathbb{R},
\end{equation*}
and 
\begin{equation*}
H_{\nu}(r)= \frac{1}{c_{H_{\nu}}} \left( (1+x_{0}) e^{r} + (1-x_{0}) e^{-r} \right) e^{- \frac{\nu r^{2}}{2 T_{\lambda}}},\ r \in \mathbb{R}.
\end{equation*}
Moreover, the constant $c_{H_{\nu}}$ can be computed explicitly
\begin{equation*}
c_{H_{\nu}} = 2 \sqrt{2 \pi T_{\lambda}/\nu} e^{\frac{T_{\lambda}}{2 \nu}},
\end{equation*}
using completion of squares. We next describe how to sample from $H_{\nu}$. Let us decompose $H_{\nu}$ into a sum of two densities $H_{\nu} = \frac{1+x_{0}}{2} H_{\nu}^{1} + \frac{1-x_{0}}{2} H_{\nu}^{2}$ with
\begin{equation*}
H_{\nu}^{1}(r) = \frac{1}{c_{H_{\nu}^{1}}} e^{r} e^{- \frac{r^{2}}{2 T_{\lambda}/\nu}} = \frac{1}{c_{H_{\nu}^{1}}} e^{\frac{T_{\lambda}}{2 \nu}} e^{- \frac{1}{2 T_{\lambda}/\nu} \left( r - \frac{T_{\lambda}}{\nu} \right)^{2}},\ r \in \mathbb{R},
\end{equation*}
and
\begin{equation*}
H_{\nu}^{2}(r) = \frac{1}{c_{H_{\nu}^{2}}} e^{-r} e^{- \frac{x^{2}}{2 T_{\lambda}/\nu}} = \frac{1}{c_{H_{\nu}^{2}}} e^{\frac{T_{\lambda}}{2 \nu}} e^{- \frac{1}{2 T_{\lambda}/\nu} \left( r + \frac{T_{\lambda}}{\nu} \right)^{2}},\ r \in \mathbb{R},
\end{equation*}
where
\begin{equation*}
c_{H_{\nu}^{1}}=c_{H_{\nu}^{2}} = \sqrt{2 \pi T_{\lambda}/\nu} e^{\frac{T_{\lambda}}{2 \nu}}.
\end{equation*}
In other words, $H_{\nu}^{1}$ is a normal distribution with mean $\mu = \frac{T_{\lambda}}{\nu}$ and variance $\sigma^{2} = T_{\lambda}/\nu$ and $H_{\nu}^{2}$ is a normal distribution with mean $\mu = - \frac{T_{\lambda}}{\nu}$ and variance $\sigma^{2} = T_{\lambda}/\nu$. Thus, we can sample from $H_{\nu}$ by sampling from $H_{\nu}^{1}$ with probability $\frac{1+x_{0}}{2}$ and sampling from $H_{\nu}^{2}$ with probability $\frac{1-x_{0}}{2}$. The next step is to compute an accept-reject sampling constant $K$ (uniform bound) such that $\sup_{r \in \mathbb{R}} \frac{h(r)}{H_{\nu}(r)} \leq K$. Observe that this is one reason we introduce the parameter $\nu$ to control the tails of $H_{\nu}$. We start with expressing
\begin{equation*}
\frac{h(r)}{H_{\nu}(r)} = \frac{c_{H_{\nu}}}{c_{h}} \left( \widetilde{q}(r) \right)^{\frac{1}{\lambda^{2}}},\ r \in \mathbb{R},
\end{equation*}
where
\begin{equation*}
\widetilde{q}(r) = (1+x_{0}) e^{r} e^{- \frac{\lambda^{2} (1-\nu) r^{2}}{2 T_{\lambda}}} + (1-x_{0}) e^{-r} e^{- \frac{\lambda^{2} (1-\nu) r^{2}}{2 T_{\lambda}}} = (1+x_{0}) \widetilde{q}_{1}(r) + (1-x_{0}) \widetilde{q}_{2}(r),
\end{equation*}
for $r \in \mathbb{R}$. With the goal of computing the uniform bound $K$, we re-write and bound $\widetilde{q}_{1}$ as
\begin{equation*}
\widetilde{q}_{1}(r) = e^{- \frac{\lambda^{2} (1-\nu)}{2 T_{\lambda}} \left( r - \frac{T_{\lambda}}{\lambda^{2} (1-\nu)}\right)^{2}} e^{\frac{(1-\nu) T_{\lambda}}{2 \lambda^{2} (1-\nu)^{2}}} \leq e^{\frac{ T_{\lambda}}{2 \lambda^{2} (1-\nu)}},\ r \in \mathbb{R},
\end{equation*}
and, similarly, we can re-write and bound $\widetilde{q}_{2}$ as
\begin{equation*}
\widetilde{q}_{2}(r) = e^{- \frac{\lambda^{2} (1-\nu)}{2 T_{\lambda}} \left( r + \frac{T_{\lambda}}{\lambda^{2} (1-\nu)}\right)^{2}} e^{\frac{(1-\nu) T_{\lambda}}{2 \lambda^{2} (1-\nu)^{2}}} \leq e^{\frac{T_{\lambda}}{2 \lambda^{2} (1-\nu)}},\ r \in \mathbb{R}.
\end{equation*}
Thus, we can estimate
\begin{equation*}
\sup_{r \in \mathbb{R}} \widetilde{q}(r) \leq 2 e^{\frac{T}{2 (1-\nu)}}
\end{equation*}
and therefore also
\begin{equation*}
\sup_{r \in \mathbb{R}} \frac{h(r)}{H_{\nu}(r)} \leq \frac{c_{H_{\nu}}}{c_{h}} 2^{\frac{1}{\lambda^{2}}} e^{\frac{T}{2 \lambda^{2} (1-\nu)}}.
\end{equation*}
To bound the quotient of the normalising constants $\frac{c_{H_{\nu}}}{c_{h}}$, we note that $h \to H_{\nu=1}$ as $\lambda \to \infty$ and we can therefore use the estimate
\begin{equation*}
c_{h} \geq \frac{1}{2} c_{H_{\nu=1}} = \sqrt{2 \pi T_{\lambda}} e^{\frac{T_{\lambda}}{2}},
\end{equation*}
which implies that
\begin{equation*}
\frac{c_{H_{\nu}}}{c_{h}} \leq 2 \sqrt{\frac{1}{\nu}} e^{\frac{(1-\nu) T_{\lambda}}{2 \nu}}.
\end{equation*}
We hence arrive at the following equality for an accept-reject sampling constant $K$
\begin{equation*}
\sup_{r \in \mathbb{R}} \frac{h(r)}{H_{\nu}(r)} \leq 2 \sqrt{\frac{1}{\nu}} e^{\frac{(1-\nu) T_{\lambda}}{2 \nu}} 2^{\frac{1}{\lambda^{2}}} e^{\frac{T}{2 \lambda^{2} (1-\nu)}} = K.
\end{equation*}

\subsection{Stochastic Nagumo equation}\label{app:Nagumo}
For the considered stochastic Nagumo equation in~\eqref{eq:Nagumo}, we first transform the SPDE in~\eqref{eq:Nagumo} into an SPDE with homogeneous Dirichlet boundary conditions. By letting
\begin{equation*}
v(t) = u(t) - 1/2,\ \forall t \in [0,T],
\end{equation*}
we have that $v(t) \in [-1/2,1/2],\ \forall t \in [0,T]$, and that $v$ satisfies the SPDE
\begin{align*}
\diff v(t) &= \Delta v(t) \diff t + \frac{1}{8} \left( 1 + 2 v(t) \right) \left( 1 - 2 v(t) \right) \left( 1 - 2 \gamma + 2 v(t) \right) \diff t \\ &+ \frac{1}{4} \left( 1 + 2 v(t) \right) \left( 1 - 2 v(t) \right) \diff W(t).
\end{align*}
Then, by letting $z(t) = 2 v(t),\ t \in [0,T]$, we have that $z(t) \in [-1,1],\ \forall t \in [0,T]$, and that $z$ satisfies the SPDE
\begin{align*}
\diff z(t) &= \Delta z(t) \diff t + \frac{1}{4} \left( 1 + z(t) \right) \left( 1 - z(t) \right) \left( 1 - 2 \gamma + z(t) \right) \diff t \\ &+ \frac{1}{2} \left( 1 + z(t) \right) \left( 1 - z(t) \right) \diff W(t)
\end{align*}
with initial value $z(0) = 2 v(0) = 2 (u(0) - 1/2)$. Moreover, $z$ satisfies homogeneous Dirichlet boundary conditions. In other words, $z$ satisfies the SPDE in~\eqref{intro:SPDE} with $f$ and $g$ given by
\begin{equation*}
f(r) = \frac{1}{4} \left( 1 + r \right) \left( 1 - r \right) \left( 1 - 2 \gamma + r \right)
\end{equation*}
and
\begin{equation*}
g(r) = \frac{1}{2} \left( 1 + r \right) \left( 1 - r \right)
\end{equation*}
subject to homogeneous Dirichlet boundary conditions. We apply the $\LTE$ scheme to approximate $z$ and then transform back to $u$.

We next provide the remaining details for the exact simulation of the Nagumo type SDE given by
\begin{equation*}
\left\lbrace
\begin{aligned}
& \diff X(t) = \frac{1}{4} \left( 1 + X(t) \right) \left( 1 - X(t) \right) \left( 1 - 2 \gamma + X(t) \right) \diff t + \frac{\lambda}{2} \left( 1 - X(t)^2 \right) \diff B(t),\ t \in (0,T] \\ 
& X(0) = x_{0} \in (-1,1),
\end{aligned}
\right.
\end{equation*}
that is needed for the $\LTE$ scheme applied to the stochastic Nagumo SPDE in~\eqref{eq:Nagumo}. Note that, for the same reason as in Section~\ref{sec:exactSim}, we scale the noise with a parameter $\lambda>0$. For the implementation of the $\LTE$ scheme we use $\lambda = \sqrt{N}$ where $N$ is the finite difference discretisation parameter. To this end, we first apply a deterministic time-change (see Section~\ref{sec:app-TC}) to transfer the constant $\widetilde{\lambda} = \frac{\lambda}{2}$ into the drift. The process $\widetilde{X}(t) = X(\widetilde{\lambda}^{-2} t)$ satisfies (up to changing the driving noise $B$) the following Itô SDEs
\begin{equation*}
\left\lbrace
\begin{aligned}
& \diff \widetilde{X}(t) = \frac{1}{4} \widetilde{\lambda}^{-2} \left( 1 + \widetilde{X}(t) \right) \left( 1 - \widetilde{X}(t) \right) \left( 1 - 2 \gamma + \widetilde{X}(t) \right) \diff t + \left( 1 - \widetilde{X}(t)^2 \right) \diff B(t),\ t \in \left(0,T_{\widetilde{\lambda}} \right] \\ 
& \widetilde{X}(0) = x_{0} \in (-1,1),
\end{aligned}
\right.
\end{equation*}
in the weak sense. In the following, we will use both $\lambda$ and $\widetilde{\lambda}$, whichever is most convenient. Since $\widetilde{\lambda} = \frac{\lambda}{2}$, which in turn implies that $\widetilde{\lambda}^{-2} = 4 \lambda^{-2}$, we have that $\widetilde{X}$ satisfies
\begin{equation}\label{eq:NagumoSDETC}
\left\lbrace
\begin{aligned}
& \diff \widetilde{X}(t) = \lambda^{-2} \left( 1 + \widetilde{X}(t) \right) \left( 1 - \widetilde{X}(t) \right) \left( 1 - 2 \gamma + \widetilde{X}(t) \right) \diff t + \left( 1 - \widetilde{X}(t)^2 \right) \diff B(t),\ t \in \left(0,T_{\widetilde{\lambda}} \right] \\ 
& \widetilde{X}(0) = x_{0} \in (-1,1),
\end{aligned}
\right.
\end{equation}
in the weak sense, where $T_{\widetilde{\lambda}} = \widetilde{\lambda}^{2} T$. Note that the above time-changed SDE with $a=1/2$ reduces to the time-changed Allen--Cahn type SDE in~\eqref{eq:ACsdeTC}, but with a different end time point. However, the Nagumo equation is usually only defined for $\gamma \in (0,1/2)$. Recall that sampling $X(\Delta t)$ starting from $X(0) = x_{0}$ is equivalent to sampling $\widetilde{X}(\widetilde{\lambda}^{2} \Delta t)$ starting from $\widetilde{X}(0) = x_{0}$. With the goal of applying exact simulation (see Section~\ref{sec:exactSim}), we compute the Lamperti transform of~\eqref{eq:NagumoSDETC} (see equation~\ref{eq:LampTrans})
\begin{equation*}
\Phi(r) = \int_{x_{0}}^{r} \frac{1}{1-w^2} \diff w = \frac{1}{2} \left( \log(1+r) - \log(1-r) - \log(1+x_{0}) + \log(1-x_{0}) \right),\ \forall r \in \D,
\end{equation*}
and the inverse of the Lamperti transform
\begin{equation*}
\Phi^{-1}(r) = \frac{(1+x_{0})e^{2r} - (1-x_{0})}{(1+x_{0})e^{2r} + (1-x_{0})},\ \forall r \in \mathbb{R}.
\end{equation*}
The transformed process $Y(t) = \Phi (\widetilde{X}(t))$ then satisfies
\begin{equation*}
\left\lbrace
\begin{aligned}
& \diff Y(t) = \alpha(Y(t)) \diff t + \diff B(t),\ t \in \left(0,T_{\widetilde{\lambda}} \right], \\ 
& Y(0) = \Phi(x_{0}) \in \mathbb{R},
\end{aligned}
\right.
\end{equation*}
with
\begin{equation*}
\alpha(r) = \left( 1 + \frac{1}{\lambda^{2}} \right) \Phi^{-1}(r) + \frac{1}{\lambda^{2}} (1- 2 \gamma),\ r \in \mathbb{R}.
\end{equation*}
By direct computations, we have the following expression for $\alpha'$
\begin{equation*}
\alpha'(r) = \left( 1 + \frac{1}{\lambda^{2}} \right) \frac{\diff}{\diff r} \Phi^{-1}(r) = \left( 1 + \frac{1}{\lambda^{2}} \right) g \left( \Phi^{-1}(r) \right) = \left( 1 + \frac{1}{\lambda^{2}} \right) \frac{4 (1-x_{0})(1+x_{0}) e^{2r}}{\left( (1+x_{0})e^{2r} + (1-x_{0}) \right)^{2}},
\end{equation*}
for $r \in \mathbb{R}$, where we used~\eqref{eq:diffPhiinv}, and for $\phi$
\begin{equation*}
\phi(r) = \frac{1}{2} \alpha'(r) + \frac{1}{2} \alpha(r)^{2} - k_{1},\ r \in \mathbb{R}.
\end{equation*}
In this case, we have that $k_{1} = 0$ and that 
\begin{equation*}
k_{2} = \frac{1}{2} \left( 1 + \frac{1}{\lambda^{2}} \right) + \frac{1}{2} \max \left( \left( 1 + \frac{2-2 \gamma}{\lambda^{2}} \right)^{2}, \left( 1 + \frac{2 \gamma}{\lambda^{2}} \right)^{2} \right).
\end{equation*}
This follows from the estimates
\begin{equation*}
0 = \left( 1 + \frac{1}{\lambda^{2}} \right) \min_{w \in [-1,1]} g(w) \leq \alpha'(r) \leq \left( 1 + \frac{1}{\lambda^{2}} \right) \max_{w \in [-1,1]} g(w) = 1 + \frac{1}{\lambda^{2}},
\end{equation*}
for every $r \in \mathbb{R}$, and
\begin{equation*}
0 \leq \alpha^{2}(r) \leq \max \left( \left( 1 + \frac{2-2 \gamma}{\lambda^{2}} \right)^{2}, \left( 1 + \frac{2 \gamma}{\lambda^{2}} \right)^{2} \right),
\end{equation*}
for every $r \in \mathbb{R}$. We next compute $A$ in~\eqref{eq:Adef}
\begin{align*}
A(r) &= \left( 1 + \frac{1}{\lambda^{2}} \right) \int_{0}^{r} \Phi^{-1}(w) \diff w + \frac{r}{\lambda^{2}} (1-2 \gamma) \\ &= \left( 1 + \frac{1}{\lambda^{2}} \right) \left( \log \left( (1+x_{0})e^{2r} + (1-x_{0}) \right) - r - \log(2) \right) + \frac{r}{\lambda^{2}} (1-2 \gamma),
\end{align*}
for $r \in \mathbb{R}$, and then use that the target density $h$ in~\eqref{eq:hdef} is proportional to 
\begin{equation*}
h(r) \propto \exp \left( A(r) - \frac{r^{2}}{2 T_{\widetilde{\lambda}}} \right) \propto \left( (1+x_{0}) e^{r} + (1-x_{0})e^{-r} \right)^{1 + \frac{1}{\lambda^{2}}} e^{-\frac{r^{2}}{2 T_{\widetilde{\lambda}}}} e^{\frac{r}{\lambda^{2}}(1-2 \gamma)},
\end{equation*}
for $r \in \mathbb{R}$, where we recall that $T_{\widetilde{\lambda}} = \widetilde{\lambda}^{2} T$. As we cannot directly sample from $h$, we use an accept-reject strategy to sample for $h$. As we will let $\frac{1}{\lambda^{2}} \to 0$, we use the proposal density $H_{\nu}$ given by
\begin{equation*}
H_{\nu}(r) \propto \left( (1+x_{0}) e^{r} + (1-x_{0}) e^{-r} \right) e^{- \frac{\nu r^{2}}{2 T_{\widetilde{\lambda}}}} e^{\frac{r}{\lambda^{2}}(1-2 \gamma)},
\end{equation*}
for some $\nu > 0$ to control the tails of $H_{\nu}$ in relation to the tails of $h$. Owing to the factors $e^{-\frac{r^{2}}{2 T_{\widetilde{\lambda}}}}$ in $h$ and  $e^{- \frac{\nu r^{2}}{2 T_{\widetilde{\lambda}}}}$ in $H_{\nu}$, $h$ and $H_{\nu}$ are probability densities and there exists normalising constants $c_{h}$ and $c_{H_{\nu}}$ such that
\begin{equation*}
h(r) = \frac{1}{c_{h}} \left( (1+x_{0}) e^{r} + (1-x_{0})e^{-r} \right)^{1 + \frac{1}{\lambda^{2}}} e^{-\frac{r^{2}}{2 T_{\widetilde{\lambda}}}} e^{\frac{r}{\lambda^{2}}(1-2 \gamma)},\ r \in \mathbb{R},
\end{equation*}
and 
\begin{equation*}
H_{\nu}(r) = \frac{1}{c_{H_{\nu}}} \left( (1+x_{0}) e^{r} + (1-x_{0}) e^{-r} \right) e^{- \frac{\nu r^{2}}{2 T_{\widetilde{\lambda}}}} e^{\frac{r}{\lambda^{2}}(1-2 \gamma)},\ r \in \mathbb{R},
\end{equation*}
Moreover, the constant $c_{H_{\nu}}$ can be computed explicitly
\begin{equation*}
c_{H_{\nu}} = (1+x_{0}) c_{H_{\nu}^{1}} + (1-x_{0}) c_{H_{\nu}^{2}},
\end{equation*}
where
\begin{equation*}
c_{H_{\nu}^{1}} = \sqrt{\frac{T_{\lambda} \pi}{2 \nu}} e^{\frac{T_{\lambda}}{8 \nu} (1 + \frac{1}{\lambda^{2}}(1-2 \gamma))^{2}}
\end{equation*}
and
\begin{equation*}
c_{H_{\nu}^{2}} = \sqrt{\frac{T_{\lambda} \pi}{2 \nu}} e^{\frac{T_{\lambda}}{8 \nu} (1 - \frac{1}{\lambda^{2}}(1-2 \gamma))^{2}}.
\end{equation*}
Note that $c_{H_{\nu}^{1}}$ and $c_{H_{\nu}^{2}}$ are the normalising constants for $H_{\nu}^{1}$ and $H_{\nu}^{2}$, respectively, defined below. The next issue is to sample from $H_{\nu}$, which we do as follows. Let us decompose $H_{\nu}$ into a sum of two densities $H_{\nu} = \frac{1+x_{0}}{2} H_{\nu}^{1} + \frac{1-x_{0}}{2} H_{\nu}^{2}$ with
\begin{equation*}
H_{\nu}^{1}(r) = \frac{1}{c_{H_{\nu}^{1}}} e^{r} e^{- \frac{\nu x^{2}}{2 T_{\widetilde{\lambda}}}} e^{\frac{r}{\lambda^{2}}(1-2 \gamma)} = \frac{1}{c_{H_{\nu}^{1}}} e^{\frac{T_{\lambda}}{8 \nu} \left( 1 + \frac{1}{\lambda^{2}} \left( 1 - 2 \gamma \right) \right)^{2}} e^{- \frac{1}{T_{\lambda}/(2 \nu)} \left( r - \frac{T_{\lambda}}{4 \nu} \left( 1 + \frac{1}{\lambda^{2}}(1-2 \gamma) \right) \right)^{2}},\ r \in \mathbb{R},
\end{equation*}
and
\begin{equation*}
H_{\nu}^{2}(r) = \frac{1}{c_{H_{\nu}^{2}}} e^{-r} e^{- \frac{\nu r^{2}}{2 T_{\widetilde{\lambda}}}} e^{\frac{r}{\lambda^{2}}(1-2 \gamma)} = \frac{1}{c_{H_{\nu}^{2}}}  e^{\frac{T_{\lambda}}{8 \nu} \left( 1 - \frac{1}{\lambda^{2}} \left( 1 - 2 \gamma \right) \right)^{2}} e^{- \frac{1}{ T_{\lambda}/(2 \nu)} \left( r + \frac{T_{\lambda}}{4 \nu} \left( 1 - \frac{1}{\lambda^{2}}(1-2 \gamma) \right) \right)^{2}},\ r \in \mathbb{R}.
\end{equation*}
In other words, $H_{\nu}^{1}$ is a normal distribution with mean $\mu = \frac{T_{\lambda}}{4 \nu} \left( 1 + \frac{1}{\lambda^{2}} (1-2 \gamma) \right)$ and variance $\sigma^{2} = T_{\lambda}/(4 \nu)$ and $H_{\nu}^{2}$ is a normal distribution with mean $\mu = - \frac{T_{\lambda}}{4 \nu} \left( 1 - \frac{1}{\lambda^{2}} (1-2 \gamma) \right)$ and variance $\sigma^{2} = T_{\lambda}/(4 \nu)$. Thus, we can sample from $H_{\nu}$ by sampling from $H_{\nu}^{1}$ with probability $\frac{1+x_{0}}{2}$ and sampling from $H_{\nu}^{2}$ with probability $\frac{1-x_{0}}{2}$. We next compute an accept-reject sampling constant $K$ (uniform bound) such that $\sup_{r \in \mathbb{R}} \frac{h(r)}{H_{\nu}(r)} \leq K$. Note that this is one reason why we introduced the parameter $\nu>0$ to control the tails of $H_{\nu}$. We start with expressing
\begin{equation*}
\frac{h(r)}{H_{\nu}(r)} = \frac{c_{H_{\nu}}}{c_{h}} \left( \widetilde{q}(r) \right)^{\frac{1}{\lambda^{2}}},\ r \in \mathbb{R},
\end{equation*}
where
\begin{equation*}
\widetilde{q}(r) = (1+x_{0}) e^{r} e^{- \frac{2 (1-\nu) r^{2}}{T}} + (1-x_{0}) e^{-r} e^{- \frac{2 (1-\nu) r^{2}}{T}} = (1+x_{0}) \widetilde{q}_{1}(r) + (1-x_{0}) \widetilde{q}_{2}(r),\ r \in \mathbb{R},
\end{equation*}
We can re-write and bound $\widetilde{q}_{1}$ as
\begin{equation*}
\widetilde{q}_{1}(r) = e^{- \frac{\lambda^{2} (1-\nu)}{2 T_{\widetilde{\lambda}}} \left( r - \frac{T_{\widetilde{\lambda}}}{\lambda^{2} (1-\nu)}\right)^{2}} e^{\frac{(1-\nu) T_{\widetilde{\lambda}}}{2 \lambda^{2} (1-\nu)^{2}}} \leq e^{\frac{ T_{\widetilde{\lambda}}}{2 \lambda^{2} (1-\nu)}},\ r \in \mathbb{R},
\end{equation*}
and, similarly, we can re-write and bound $\widetilde{q}_{2}$ as
\begin{equation*}
\widetilde{q}_{2}(r) = e^{- \frac{\lambda^{2} (1-\nu)}{2 T_{\widetilde{\lambda}}} \left( r + \frac{T_{\widetilde{\lambda}}}{\lambda^{2} (1-\nu)}\right)^{2}} e^{\frac{(1-\nu) T_{\widetilde{\lambda}}}{2 \lambda^{2} (1-\nu)^{2}}} \leq e^{\frac{ T_{\widetilde{\lambda}}}{2 \lambda^{2} (1-\nu)}},\ r \in \mathbb{R}.
\end{equation*}
Thus,
\begin{equation*}
\sup_{r \in \mathbb{R}} \widetilde{q}(r) \leq 2 e^{\frac{T_{\widetilde{\lambda}}}{2 \lambda^{2} (1-\nu)}} = 2 e^{\frac{T}{8 (1-\nu)}}
\end{equation*}
and therefore also
\begin{equation*}
\sup_{r \in \mathbb{R}} \frac{h(r)}{H_{\nu}(r)} \leq \frac{c_{H_{\nu}}}{c_{h}} 2^{\frac{1}{\lambda^{2}}} e^{\frac{T}{8 \lambda^{2} (1-\nu)}}.
\end{equation*}
To bound the quotient of the constants $\frac{c_{H_{\nu}}}{c_{h}}$, we note that $h \to H_{\nu=1}$ as $\lambda \to \infty$ and we can estimate
\begin{equation*}
c_{h} \geq \frac{1}{2} c_{H_{\nu=1}} = \frac{1+x_{0}}{2} c_{H_{\nu=1}^{1}} + \frac{1-x_{0}}{2} c_{H_{\nu=1}^{2}},
\end{equation*}
which, in turn, implies that
\begin{equation*}
\frac{c_{H_{\nu}}}{c_{h}} \leq 2 \frac{c_{H_{\nu}}}{c_{H_{\nu=1}}}.
\end{equation*}
Thus, finally, we arrive at
\begin{equation*}
\sup_{r \in \mathbb{R}} \frac{h(r)}{H_{\nu}(r)} \leq \frac{c_{H_{\nu}}}{c_{H_{\nu=1}}} 2^{1+\frac{1}{\lambda^{2}}} e^{\frac{T}{8 \lambda^{2} (1-\nu)}} = K
\end{equation*}
and
\begin{equation*}
\frac{c_{H_{\nu}}}{c_{H_{\nu=1}}} = \frac{(1+x_{0}) c_{H_{\nu}^{1}}+ (1-x_{0}) c_{H_{\nu}^{2}}}{(1+x_{0}) c_{H_{\nu=1}^{1}}+ (1-x_{0}) c_{H_{\nu=1}^{2}}}
\end{equation*}

\subsection{SIS SPDE}\label{app:SIS}
We first transform the SIS SPDE in~\eqref{eq:SIS} into an SPDE with homogeneous Dirichlet boundary conditions. To this end, if we let 
\begin{equation*}
v(t) = u(t) - 1/2,\ \forall t \in [0,T],
\end{equation*}
then $v(t) \in [-1/2,1/2]$,for $t \in [0,T]$, and $v$ satisfies the SPDE
\begin{align*}
\diff v(t) &= \Delta v(t) \diff t + \left( 1/2 + v(t) \right) \left( 1/2 - v(t) \right) \diff t \\ &+ \left( 1/2 + v(t) \right) \left( 1/2 - v(t) \right) \diff W(t).
\end{align*}
Second, if we also let $z(t) = 2 v(t),\ \forall t \in [0,T]$, then $z(t) \in [-1,1],\ \forall t \in [0,T]$, and $z$ satisfies the SPDE
\begin{equation*}
\diff z(t) = \Delta z(t) \diff t + \frac{1}{2} \left( 1 + z(t) \right) \left( 1 - z(t) \right) \diff t + \frac{1}{2} \left( 1 + z(t) \right) \left( 1 - z(t) \right) \diff W(t)
\end{equation*}
with initial value $z(0) = 2 v(0) = 2 \left( u(0) - 1/2 \right)$. Moreover, $z$ satisfies homogeneous Dirichlet boundary conditions. Hence, $z$ satisfies the SPDE in~\eqref{intro:SPDE} with
\begin{equation*}
f(r) = \frac{1}{2} (1+r) (1-r)
\end{equation*}
and
\begin{equation*}
g(r) = \frac{1}{2} (1+r)(1-r)
\end{equation*}
subject to homogeneous Dirichlet boundary conditions. We apply the $\LTE$ scheme to approximate $z$ and the transform back to $u$.

We now provide the remaining details for the exact simulation of the SIS SDE given by
\begin{equation*}
\left\lbrace
\begin{aligned}
& \diff X(t) = \frac{1}{2} \left( 1 + X(t) \right) \left( 1 - X(t) \right) + \frac{\lambda}{2} \left( 1 + X(t) \right) \left( 1 - X(t) \right) \diff B(t),\ t \in (0,T] \\ 
& X(0) = x_{0} \in (-1,1),
\end{aligned}
\right.
\end{equation*}
that is needed for the $\LTE$ scheme applied to the SIS SPDE in~\eqref{eq:SIS}. As usual, we scale the noise with a parameter $\lambda>0$ and we are interested in the case $\lambda = \sqrt{N}$ where $N$ is the finite difference discretisation parameter. We first apply a deterministic time-change (see Section~\ref{sec:app-TC}) to transfer the constant $\widetilde{\lambda} = \frac{\lambda}{2}$ into the drift. The process $\widetilde{X}(t) = X(\widetilde{\lambda}^{-2} t)$ then satisfies (up to changing the driving noise $B$)
\begin{equation}\label{eq:SDEsdeTC}
\left\lbrace
\begin{aligned}
& \diff \widetilde{X}(t) = \frac{1}{2} \widetilde{\lambda}^{-2} \left( 1 + \widetilde{X}(t) \right) \left( 1 - \widetilde{X}(t) \right) \diff t + \left( 1 + \widetilde{X}(t) \right) \left( 1 - \widetilde{X}(t) \right) \diff B(t),\ t \in \left(0,T_{\widetilde{\lambda}} \right], \\ 
& \widetilde{X}(0) = x_{0} \in (-1,1),
\end{aligned}
\right.
\end{equation}
in the weak sense, where $T_{\widetilde{\lambda}} = \widetilde{\lambda}^{2} T$. We recall that sampling $X(\Delta t)$ starting from $X(0) = x_{0}$ is equivalent to sampling $\widetilde{X}(\widetilde{\lambda}^{2} \Delta t)$ starting from $\widetilde{X}(0) = x_{0}$ and that we are interested in the case $T = \Delta t$ and $T_{\widetilde{\lambda}} = \widetilde{\lambda}^{2} \Delta t = \frac{\lambda^{2}}{4} \Delta t$. To the end of applying exact simulation, we compute the Lamperti transform of~\eqref{eq:SDEsdeTC} (see equation~\eqref{eq:LampTrans})
\begin{equation*}
\Phi(r) = \int_{x_{0}}^{r} \frac{1}{1-w^2} \diff w = \frac{1}{2} \left( \log(1+r) - \log(1-r) - \log(1+x_{0}) + \log(1-x_{0}) \right),\ r \in \D,
\end{equation*}
and the inverse of the Lamperti transform
\begin{equation*}
\Phi^{-1}(r) = \frac{(1+x_{0})e^{2r} - (1-x_{0})}{(1+x_{0})e^{2r} + (1-x_{0})},\ r \in \mathbb{R}.
\end{equation*}
The transformed process $Y(t) = \Phi (\widetilde{X}(t))$ then satisfies
\begin{equation*}
\left\lbrace
\begin{aligned}
& \diff Y(t) = \alpha(Y(t)) \diff t + \diff B(t),\ t \in \left(0,T_{\widetilde{\lambda}} \right], \\ 
& Y(0) = \Phi(x_{0}) \in \mathbb{R},
\end{aligned}
\right.
\end{equation*}
with drift coefficient
\begin{equation*}
\alpha(r) = 2 \lambda^{-2} + \Phi^{-1}(r),\ \forall r \in \mathbb{R}.
\end{equation*}
We can compute $\alpha'$ by direct differentiation
\begin{equation*}
\alpha'(r) = \frac{\diff}{\diff r} \Phi^{-1}(r) = g \left( \Phi^{-1}(r) \right) = \frac{4 (1-x_{0})(1+x_{0}) e^{2r}}{\left( (1+x_{0})e^{2r} + (1-x_{0}) \right)^{2}},\ r \in \mathbb{R},
\end{equation*}
where we also used~\eqref{eq:diffPhiinv}, and we recall the definition of $\phi$
\begin{equation*}
\phi(r) = \frac{1}{2} \alpha'(r) + \frac{1}{2} \alpha(r)^{2} - k_{1},\ r \in \mathbb{R},
\end{equation*}
In this case, we have that $k_{1}=0$ and that
\begin{equation*}
k_{2} = \frac{1}{2} \left( 1 + \frac{2}{\lambda^{2}} \right)^{2} + \frac{1}{2}.
\end{equation*}
The above follows from the estimates
\begin{equation}\label{eq:SIS-alphaalphap}
0 \leq \alpha'(r) \leq 1,\ 0 \leq \alpha^{2}(r) \leq \left( 1 + \frac{2}{\lambda^{2}} \right)^{2},
\end{equation}
both for every $r \in \mathbb{R}$. The first inequality in~\eqref{eq:SIS-alphaalphap} follows from
\begin{equation*}
0 \leq \alpha'(r) = g(\Phi^{-1}(r)) \leq \max_{w \in [-1,1]} g(w) \leq 1,
\end{equation*}
for every $r \in \mathbb{R}$, and the second inequality in~\eqref{eq:SIS-alphaalphap} follows from
\begin{equation*}
0 \leq \alpha^{2}(r) = \left( 2 \lambda^{-2} + \Phi^{-1}(r) \right)^{2} \leq \left( 1 + \frac{2}{\lambda^{2}} \right)^{2},
\end{equation*}
for every $r \in \mathbb{R}$. We now compute $A$ in~\eqref{eq:Adef}
\begin{equation*}
A(r) = 2 \lambda^{-2} r + \int_{0}^{r} \Phi^{-1}(w) \diff w = 2 \lambda^{-2} e + \log \left( (1+x_{0})e^{2e} + (1-x_{0}) \right) - e - \log(2),
\end{equation*}
for $r \in \mathbb{R}$, and then use that the target density $h$ (see~\eqref{eq:hdef}) is proportional to 
\begin{equation}
h(r) \propto \exp \left( A(r) - \frac{r^{2}}{2 T_{\widetilde{\lambda}}} \right) \propto \left( (1+x_{0}) e^{r} + (1-x_{0})e^{-r} \right) e^{2 \lambda^{-2} r} e^{-\frac{r^{2}}{2 T_{\widetilde{\lambda}}}},
\end{equation}
for $r \in \mathbb{R}$. In this case, we can sample from $h$ (without applying an accept-reject strategy). We decompose $h$ as
\begin{equation}\label{eq:SIShexpr}
h(r) = \frac{1+x_{0}}{2} q_{1}(r) + \frac{1-x_{0}}{2} q_{2}(r),\ r \in \mathbb{R},
\end{equation}
where $q_{1}$ and $q_{2}$ are normal densities with parameters
\begin{equation*}
q_{1} \sim N \left(\mu = \frac{T_{\lambda}}{4} \left( 1 + \frac{2}{\lambda^{2}} \right),\ \sigma^{2} = \frac{T_{\lambda}}{4} \right)
\end{equation*}
\begin{equation*}
q_{2} \sim N \left(\mu = - \frac{T_{\lambda}}{4} \left( 1 - \frac{2}{\lambda^{2}} \right),\ \sigma^{2} = \frac{T_{\lambda}}{4} \right).
\end{equation*}
Note that the decomposition in~\eqref{eq:SIShexpr} follows from completion of squares.  Thus, we can sample from $h$ by sampling from $q_{1}$ with probability $\frac{1+x_{0}}{2}$ and from $q_{2}$ with probability $\frac{1-x_{0}}{2}$.
\end{appendix}

\section*{Acknowledgements}
This work is partially supported by the Swedish Research Council (VR) (David Cohen's project nr. $2018-04443$). The computations were enabled by resources provided by Chalmers e-Commons at Chalmers and by resources provided by the National Academic Infrastructure for Supercomputing in Sweden (NAISS), partially funded by the Swedish Research Council through grant agreement no. 2022-06725.

\bibliographystyle{abbrv}
\bibliography{Boundary-preserving_weak_approximation_for_some_semilinear_stochastic_partial_differential_equations.bib}

\end{sloppypar}
\end{document}